\documentclass[12pt,a4paper]{amsart}
\usepackage{tikz}
\usetikzlibrary{decorations.markings}
\usetikzlibrary{cd}
\usetikzlibrary{patterns}
\usetikzlibrary{shapes}

\tikzset{anchorbase/.style={baseline={([yshift=-0.5ex]current bounding 
box.center)}}}
\tikzset{wipe/.style={white,line width=4pt}}
\usepackage{ctable}
\usepackage{stackengine}
\usepackage{hyperref}
\usepackage{latexsym,fullpage,amsfonts,amssymb,amsmath,amscd,graphics,epic}
\usepackage[all]{xy}
\usepackage{amssymb,amsbsy,amsthm,amsxtra}
\usepackage{amsthm}
\usepackage{mathrsfs}
\usepackage{url}
\usepackage{bbm}
\usepackage{wasysym}
\usepackage{enumitem}
\usepackage{framed}

\usepackage[vcentermath]{youngtab}%
\theoremstyle{plain}
\newtheorem*{theorem*}{Theorem}
\newtheorem*{remark*}{Remark}
\newtheorem*{example*}{Example}
\newtheorem{lemma}{Lemma}[subsection]
\newtheorem{proposition}[lemma]{Proposition}
\newtheorem{corollary}[lemma]{Corollary}
\newtheorem{theorem}[lemma]{Theorem}

\newtheorem*{conjecture*}{Conjecture}

\newtheorem{introtheorem}{Theorem}

\newtheorem{introprob}[introtheorem]{Problem}

\theoremstyle{definition}
\newtheorem{definition}[lemma]{Definition}

\newtheorem{example}[lemma]{Example}

\theoremstyle{remark}
\newtheorem{remark}[lemma]{Remark}

\newtheorem{notation}[lemma]{Notation}

\newcommand\xrowht[2][0]{\addstackgap[0.5\dimexpr#2\relax]{\vphantom{#1}}}

\newcommand{\Hom}{\operatorname{Hom}}

\newcommand{\triv}{{\mathbbm 1}}

\newcommand{\id}{\operatorname{Id}}

\newcommand{\Ker}{\operatorname{Ker}}

\newcommand{\F}{{\mathbb F}}

\newcommand{\End}{\operatorname{End}}

\newcommand{\C}{{\mathbb C}}
\newcommand{\Z}{{\mathbb Z}}

\newcommand{\lam}{{\lambda}}

\newcommand{\abs}[1]{\left|{#1}\right|}

\newcommand{\sgn}{\mathtt{sgn}}

\newcommand{\InnaA}[1]{#1} 
\newcommand{\InnaB}[1]{#1}
\newcommand{\InnaC}[1]{#1}
\newcommand{\comment}[1]

\def\quotient#1#2{%
    \raise1ex\hbox{$#1$}\Big/\lower1ex\hbox{$#2$}%
}

\begin{document}

\date{\today}
\title{McKay trees}
 \author{Avraham Aizenbud, Inna Entova-Aizenbud}
\address{Avraham Aizenbud, Dept. of Mathematics, Weizmann Institute of Science, Rehovot,
Israel.}
\email{aizenr@gmail.com}
\address{Inna Entova-Aizenbud, Dept. of Mathematics, Ben Gurion University of the Negev,
Beer-Sheva,
Israel.}
\email{entova@bgu.ac.il}
\begin{abstract}
Given a finite group $G$ and its representation $\rho$, the corresponding McKay graph is a graph $\Gamma(G,\rho)$ whose vertices are the irreducible representations of $G$; the number of edges between two vertices $\pi,\tau$ of $\Gamma(G,\rho)$ is $\dim \Hom_G(\pi\otimes \rho, \tau) $. The collection of all McKay graphs for a given group $G$ encodes, in a sense, its character table. Such graphs were also used by McKay to provide a bijection between the finite subgroups of $SU(2)$ and the affine Dynkin diagrams of types $A, D, E$, the bijection given by considering the appropriate McKay graphs. 

In this paper, we classify all (undirected) \InnaB{trees} which are McKay graphs of finite groups \InnaB{and describe the corresponding pairs $(G,\rho)$}; this classification turns out to be very concise.

Moreover, \InnaB{we give a partial classification of McKay graphs which are forests, and construct some non-trivial examples of such forests.}
 \end{abstract}

\maketitle
\setcounter{tocdepth}{3}
\section{Introduction}

\subsection{Definition of the McKay graph}
For a finite group $G$ and its complex representation $\rho$, we consider the McKay graph $\Gamma(G, \rho)$: its vertices are given by the set $Irr(G)$ of irreducible complex representations of $G$, and the number of edges $N_{\pi,\tau}$ between two vertices $\pi,\tau \in Irr(G)$ of $\Gamma(G,\rho)$ is $\dim \Hom_G (\pi\otimes \rho, \tau) $. In general, this is a directed graph.

We will consider undirected graphs as a special case of directed graphs. If $\rho$ is self-dual then $N_{\pi,\tau}=N_{\tau,\pi}$ and the adjacency matrix of the graph $\Gamma(G,\rho)$ is symmetric. In this case we will consider the McKay graph as an undirected graph, with an undirected edge corresponding to a pair of directed edges of opposite directions.

\subsection{Main results}
We work over the base field $\C$. In this paper we consider McKay graphs which are (undirected) forests: these are undirected graphs without closed paths (in a path, no repetitions of edges are allowed). In particular, in such a graph there is at most one edge between each $2$ vertices, and no edges from a vertex to itself.
\begin{introtheorem}[See \InnaB{Proposition \ref{prop:forest},} Theorem \ref{thrm:forest}]\label{introthm:forests}
Let $G$ be a finite group and $\rho$ its representation.

Assume that the McKay graph $\Gamma(G,\rho)$ is a forest. Then we have: 
    \begin{enumerate}
        \item $\rho$ is irreducible and self-dual.
        \item Each of these components  is isomorphic to one of the following:
\begin{enumerate}
    \item An affine Dynkin diagram of type $\tilde{D}_{n}$, $n\geq 4$.
\item An affine Dynkin diagram of type $\tilde{E}_n$ (for $n=6,7,8$).
\item  A ``hedgehog'': a tree with $4^n$ leaves and one vertex \InnaA{connected to all of these leaves}\footnote{\InnaA{For $n=0$, we obtain a graph with $2$ vertices and an unoriented edge between them.}}, \InnaB{where} $n\geq 0$.
\end{enumerate}
\InnaA{Furthermore, if one of the components is a ``hedgehog'' with $4^n$ leaves, $n\neq 1$, then \InnaB{it is isomorphic to} the remaining components.}
\end{enumerate}
\end{introtheorem}
Our second result concerns McKay graphs which are trees (i.e. connected forest graphs):
\begin{introtheorem}[See Theorem \ref{thrm:main}]\label{introthm:trees}
Assume $\Gamma(G,\rho)$ is a tree. Then the pair $(G,\rho)$ is isomorphic to one of the following pairs: 
\begin{enumerate}
    \item $G$ is a dihedral group of order divisible by $4$, and $\rho$ is the complexification of its natural $2$-dimensional real representation.
    
\item $G \subset SU(2)$ is a binary polyhedral group of type $D$ or $E$ and $\rho$ is the natural $2$-dimensional representation of $G$.

\item $G$ is an extra special group of order $2^{1+2n}$ for some $n\geq 0$, i.e. a central extension of $(\Z/2\Z)^{2n}$ by $\Z/2\Z$. In that case, $\rho$ is the unique irreducible representation of $G$ of dimension $2^n$.
\end{enumerate} 
\end{introtheorem}

\begin{remark}\label{intrormk:principal_comp}
The connected component of the McKay graph $\Gamma(G,\rho)$ containing the trivial representation is called the principal component of $\Gamma(G,\rho)$ and it is itself a McKay graph for the group $G/Ker(\rho)$ and its faithful representation $\rho$. 

If $\Gamma(G,\rho)$ is a forest, our second result describes the pair $(G/Ker(\rho), \rho)$ up to isomorphism.
\end{remark}

\InnaA{
Finally, we construct some examples of McKay graphs which are forests with non-isomorphic connected components. Constructing such forests where all the connected components are isomorphic to a given tree from the list above is very easy: for any $n\geq 2$, let $(G, \rho)$ be as in Theorem \ref{introthm:trees}, and consider the McKay graph $\Gamma(G\times \Z/n\Z, \rho \boxtimes \triv)$; this graph is the union of $n$ copies of the tree $\Gamma(G, \rho)$. 

Constructing and classifying forests whose connected components have different sizes seems to be a much more challenging problem. We give several examples of such forests in Section \ref{ssec:forest_examples}.

}

\subsection{Background and motivation}
Let $G$ be a finite group. Then one can reconstruct the isomorphism class of $G$ from the monoidal category of representations of $G$. Therefore, while classifying such monoidal categories is a central and interesting problem, in full generality it is as \InnaB{complex} as the problem of classifying finite groups. A numerical shadow of the monoidal category of representations of $G$ is the character table of $G$. 
It is well-known that this shadow is not enough to reconstruct the group $G$ (for instance, the dihedral group $Dih_4$ and the quaternion group $Q_8$ have the same character table). Yet even the question of describing all \InnaB{possible} character tables still looks very difficult in general. The McKay graph is a combinatorial way to describe a piece of this character table. In fact, the data in the character table of a group $G$ is equivalent to the data of all the McKay graphs $\Gamma(G, \rho)$ where $\rho$ runs over the irreducible representations of $G$ and the labeling of the vertices in the graphs is fixed.
Another motivation to consider McKay graphs is the fact that they are spectral analogues of the Cayley graphs.

So the starting point for this paper is the following:
\begin{introprob}
\mbox{}
\begin{enumerate}
    \item Which (directed) graphs $\Gamma$ are McKay graphs for some finite group $G$ and some representation $\rho$ of $G$? 
    \item Given a finite (directed) graph $\Gamma$ which is a McKay graph, what are the different pairs $(G,\rho)$ such that $\Gamma=\Gamma(G,\rho)$?      
\end{enumerate}
\end{introprob}
It is not clear whether any of these questions can be answered in full generality. We do not expect a comprehensive answer to both of them, since this is equivalent to describing all finite groups and their representations. However, we believe that studying special cases of this problem is illuminating. 

In the present paper, we \InnaB{give a partial} answer \InnaB{to} the first question in the case when $\Gamma$ is an (undirected) forest and \InnaB{answer both questions} in the case when $\Gamma$ is an (undirected) tree.

\subsection{Related works}\label{ssec:intro_rel_works}
In \cite{McKay}, McKay classified McKay graphs with spectral radius $2$; this leads to the ADE classification of the finite subgroups of $SU(2)$, and can be adapted to the classification of pairs $(G,\rho)$ where $\rho$ is a $2$-dimensional self-dual faithful representation of $G$ (see Proposition \ref{prop:McKay_classification}). 

\subsection{Sketch of the proof}

\subsubsection{The case of a tree graph}
One of the main methods to study a McKay graph $\Gamma(G,\rho)$ is to observe that if we consider its adjacency matrix as an operator from the linear span of $Irr(G)$ to itself, then this operator can be realized as the multiplication by $\chi_\rho$ when we identify the linear span of $Irr(G)$ with the space of class functions on $G$.

Recall that a graph $\Gamma$ is a tree iff the following holds:
\begin{enumerate}[label=(\alph*)]
    \item it is undirected,
    \item it is connected,
    \item the number of its vertices exceeds by 1 the number of its (undirected) edges.
\end{enumerate}
If $\Gamma=\Gamma(G,\rho)$ is a McKay graph,  then the first condition is equivalent to the fact that $\rho$ is self-dual, and the second condition is equivalent to the fact that $\rho$ is faithful. So it is left to exploit the third condition. In order to do this, we prove (in Corollary \ref{cor:trace}) the following formula for a general undirected McKay graph $\Gamma = \Gamma(G, \rho)$:
\begin{equation}\label{eq:dim.formula}
2\abs{\{\text{undirected edges in } \Gamma\}}=tr(A_\Gamma^2)=\sum_{[x]\in G//G}\chi_\rho(x)^2=\sum_{[x]\in G//G}\dim \End_{C(x)}\left( \rho\downarrow_{C(x)}\right),
\end{equation}
Here
\begin{itemize}
    \item $A_\Gamma$ is the adjacency matrix of the graph $\Gamma$.
    \item $x$ ranges over (a set of representatives of) the conjugacy classes of $G$.
    \item $C(x)$ denotes the centralizer of $x$.
\end{itemize}

We use this formula in order to deduce that if $\Gamma$ is a forest then $\rho$ is irreducible and for any non-central $x \in G$, we have (see the proof of Proposition \ref{prop:forest}): \begin{equation}\label{eq:dim2}
\dim(\End(\rho\downarrow_{C(x)}))=2.    
\end{equation}

We finish the argument in the proof of Theorem \ref{introthm:trees} by considering the following cases:
\begin{enumerate}[label=(\alph*)]
    \item {\it $C(x)$ is abelian for some $x\in G$.} 
    In this case it is easy to see that $\rho$ is 2-dimensional, 
    so we can use McKay's classification (see Section \ref{ssec:intro_rel_works}).
    \item {\it $C(x)$ is non-abelian for any $x\in G$.}
    In this case we prove that $G$ is an extra special $2$-group (i.e. a  $2$-group of unipotent depth $2$ and center of order $2$). Since such groups are classified (see \cite{Aschbacher, Diaconis}), this completes the argument behind the proof of Theorem \ref{introthm:trees}.
\end{enumerate}
\subsubsection{The case of forests}
\InnaB{We give below an overview of the main ingredients in the proof of Theorem \ref{introthm:forests}.}
\InnaA{
\begin{enumerate}[label=(\alph*)]
    \item Let $N=\Ker(\rho)$ and consider the action of $G$ on $Irr(N)$ induced by the conjugationaction of $G$ on $N$. It is well-known that the set of connected components of a McKay graph $\Gamma(G,\rho)$ is naturally indexed by the set of $G$-orbits $Irr(N)//G$. The indexing is obtained by restricting a representation $\mu\in  \Gamma(G,\rho)$ to $N$ (see Lemma \ref{lem:connected_comp_Browne}).
    \item Let $\tau\in Irr(N)$ and let $T\in Irr(N)//G$ be the $G$-orbit of $\tau$. Let $\Gamma_T(G,\rho)$ be the connected component of $\Gamma(G,\rho)$ corresponding to $T$. Then we have (see Lemma \ref{lem:conn_comp_sum_of_sq}):
    
    \begin{equation}\label{eq:intro_sum_of_sq}
     \sum_{\mu\in \Gamma_T(G,\rho)} (\dim \mu)^2=|G/N||T|(\dim\tau)^2.
    \end{equation}
\item Using Remark \ref{intrormk:principal_comp}, we conclude that the principal component $\Gamma_\triv$ is either an affine Dynkin diagram of types $D$ or $E$ (in that case, $\dim(\rho)=2$), or a hedgehog with $4^n$ spines (in that case, $\dim(\rho)=2^n$). 
\item If $\dim(\rho)=2$, we use McKay's classification of graphs with spectral radius $2$ to state that all the connected components will be affine Dynkin diagrams of types $D$ or $E$.
\item If the principal component is a hedgehog, we use  \eqref{eq:intro_sum_of_sq} to show that $|T|=1$ for any $T\in Irr(N)//G$, and conclude that all the connected components in this case are isomorphic to the principal one.
\end{enumerate}
}
\InnaB{We then give some examples of McKay graphs which are forests with non-isomorphic connected components. In a follow-up paper, we plan to  \InnaC{investigate this question further.}}
\subsection{Structure of the paper}
In Section \ref{sec:prelim} we give the required preliminaries on groups and McKay graphs, including the definition of the extraspecial groups (see Section \ref{ssec:extra_special}) and the McKay correspondence between affine Dynkin diagrams and finite subgroups of $SU(2)$ (see Section \ref{ssec:mckay_class}). In Section \ref{sec:Aux} we prove some auxilary results, such as a criterion for a McKay graph \InnaB{to be  connected bipartite} and the formula $$\forall k\geq 1, \;\;
\dim \Hom_G(\rho^{\otimes k}, \C[G]^{adj}) = \abs{\{\text{circuits of length } k \text{ in } \Gamma \}},  $$
which we use later on.
In Section \ref{sec:trees} we prove Theorem \ref{introthm:trees}. In Section \ref{sec:forests} we prove Theorem \ref{introthm:forests} and construct some non-trivial examples of forests in Subsection \ref{ssec:forest_examples}. 
\subsection{Acknowledgements}
A.A. was partially supported by ISF grant 249/17 and a Minerva foundation grant. I.E.-A. was supported by the ISF grant 711/18.

\section{Preliminaries and notation}\label{sec:prelim}
The base field throughout the paper is $\C$.
\subsection{Notation}
\begin{notation}
 Let $G$ be a finite group.
 \begin{itemize}
 \item We denote by $1_G$ its unit element, by $Z(G)$ its center, by $G//G$ the set of conjugacy classes of $G$, and by $C(g)$ the centralizer of $g\in G$.
  \item We denote by $Irr(G)$ the set of its irreducible representations and by $\triv$ its trivial representation. 
  \item We denote by $\C[G]^{adj}$ the conjugation representation of $G$ on its group algebra.
  \item For any representation $\rho$ of $G$, we denote by $\chi_\rho: G\to \C$ its character, where $\chi_\rho(g):=tr (\rho(g))$. By abuse of notation, we use the same notation for the corresponding function $\chi_\rho: G//G \to \C$.
 \item We say that a group $G$ is $p$-torsion, where $p$ is prime, if $g^p = 1_G$ for each $g \in G$.
 \item Given a subgroup $H < G$, a $G$-representation $\rho$ and an $H$-representation $\tau$, we denote by $\rho \downarrow_H$ the restriction of $\rho$ to $H$ and by $\tau \uparrow^G_H$ the induction of $\tau$ to $G$.
 \end{itemize}

\end{notation}

\begin{notation}
 \InnaB{For $n\geq 2$, we will denote by $C_n$ the cyclic group of order $n$. For prime $p$, the elementary abelian group $p$-group $C_p^{\times k}$ will sometimes be denoted $\mathbb{F}_p^k$.}
\end{notation}

\begin{notation}
 \InnaB{For $n\geq 3$,} we will denote by $\mathrm{Dih}_n$ the dihedral group of order $2n$ (group of symmetries of the regular $n$-gon). \InnaB{We will also denote by $\mathrm{Dih}_2$ the group $C_2 \times C_2$ (considered as the group of symmetries of a digon).}
\end{notation}
\comment{
\begin{example}
  \InnaB{The group $\mathrm{Dih}_2:=C_2 \times C_2$ is the group of symmetries of the digon, a ``degenerate polygon" with $2$ vertices. A digon is given by two vertices connected by symmetric arcs on both sides of the line connecting the vertices (e.g. a circle with two marked antipodal vertices):
\begin{equation*}
\begin{tikzpicture}[anchorbase,scale=1]
       \draw (-1,0) to[out=40,in=140]  (1,0);
       \draw (-1,0) to[out=-40,in=-140]  (1,0);
       \node at (-1,0) {$\bullet$};
       \node at (1,0) {$\bullet$};
\end{tikzpicture}
 \end{equation*}}
\end{example}}

\begin{notation}
 Let $\Gamma$ be a directed graph.  
 \begin{itemize}
 \item We denote by $X(\Gamma)$ the set of vertices of $\Gamma$. 
 \item We denote by $\mathtt{N}(x)$ the multiset of neighbors of a vertex $x$ in $\Gamma$: those are $y\in X(\Gamma)$ with an edge $x\to y$ in $\Gamma$ (if there are several edges $x \to y$, then $y$ appears with appropriate multiplicity in $\mathtt{N}(x)$.
 \item A loop is an edge connecting a vertex to itself). 
 \item A graph $\Gamma$ is called {\it simply-laced} if for each $x, y \in X(\Gamma)$ it contains at most one edge $x \to y$.
 \item A {\it walk} in $\Gamma$ is a sequence of edges $(e_0, \ldots, e_s)$ such that the end of $e_i$ is the beginning of $e_{i+1}$ for each $i<s$ (repetitions of edges are allowed!). A walk is called a {\it circuit} if the end of $e_s$ is the beginning of $e_{0}$. A {\it path} is a walk without repetitions of edges.
 \item Given a pair of edges in opposite directions (or a loop), we will usually draw a single {\it undirected} edge instead. If all the edges in the graph are undirected, we will say that this graph is {\it undirected}. 
 \item We say that $\Gamma$ is an {\it undirected tree} if $\Gamma$ is undirected and when considered as an undirected graph, it contains no circuits (in particular, it is simply-laced and without loops!).
 \end{itemize}
 
\end{notation}
\begin{notation}
 Given a graph $\Gamma$ with vertex set $X$, we denote by $A_{\Gamma}$ the {\it adjacency matrix} of this graph. That is, $A_{\Gamma}$ is an $\abs{X} \times \abs{X}$ matrix with integer coefficients whose rows and columns are indexed by the set $X$. The $(x,y) \in X \times X$ entry of $A_{\Gamma}$ is the number of edges in $\Gamma$ from $x$ to $y$. The matrix $A_{\Gamma}$ naturally defines a linear endomorphism of the space of functions $X \to \C$, and we will identify $A_{\Gamma}$ with this endomorphism. 
\end{notation}

\comment{
\InnaA{
\begin{notation}
Let $\Gamma$ be a graph. We denote by $Spec(\Gamma)$ the spectrum (multiset of eigenvalues) of the adjancency matrix $A_\Gamma$.
\end{notation}
}}

\begin{remark}
 A graph $\Gamma$ is undirected if and only if $A_{\Gamma}$ is symmetric.
\end{remark}

\subsection{Preliminaries on extra special groups}\label{ssec:extra_special}
A special class of finite groups which appear in this paper are extra special groups. For an introduction to extra special groups, see \cite{Aschbacher} and the appendix of \cite{Diaconis}.

\begin{definition}
 Let $p$ be a prime number. A finite group $G$ is called an {\it extra special $p$-group} if $Z(G) \cong C_p$ and $G/Z(G) \cong (C_p)^N \InnaB{= \F_p^N}$ for some $N\geq 0$. 
\end{definition}
\begin{example}\label{ex:extra_special_groups}
\mbox{}
\begin{itemize}
 \item There is just one extra special group of order $p$, up to isomorphism, which is $C_p$. 
 \item There are two examples of extra special $p$-groups of order $p^3$: the 
Heisenberg group $Heis(p)$ given by upper-triangular $3\times 3$ matrices over the field 
$\mathbb{F}_p$ with $1$'s on the diagonal, and the semidirect product $C_{p^2} 
\rtimes C_p$ where $C_p$ acts non-trivially on $C_{p^2}$. These are non-isomorphic for $p>2$. 
 \item For $p=2$ the above examples coincide: we have $Heis(2)\cong C_4 \rtimes C_2 \cong \mathrm{Dih}_4$. Another (non-isomorphic) extra special group of 
order $8$ is the quaternion group $Q_8$.
 \item Given a finite-dimensional vector space $V$ over the field $\F_p$, we can 
define the generalized Heisenberg group $Heis(V)$ as follows: the underlying set 
is $V \times V^* \times \F_p$ and the multiplication is given by $$(\vec{a}, 
\vec{b}, c)(\vec{a'}, \vec{b'}, c') = (\vec{a} +\vec{a}, \vec{b}+\vec{b'}, c + 
c' + \vec{b'}\cdot\vec{a})$$ where $-\cdot -$ denotes the obvious pairing $V^* 
\times V \to \F_p$. This group $Heis(V)$ is an extra special $p$-group of order 
$p^{1+2\dim V}$.
\end{itemize}

\end{example}

It turns out that such groups exist only if $N$ is even, i.e. $N=2n$ for some 
$n\geq 0$; moreover, for each power $p^{1+2n}$ ($n\geq 1$) there exist precisely 
two isomorphism classes of extra special $p$-groups of order $p^{1+2n}$. 

For $p>2$, any extra special $p$-group of order at least $p^3$ is a central 
product of several copies of the two groups of order $p^3$ appearing in Example 
\ref{ex:extra_special_groups}. If all of the $n$ factors appearing in the 
central product are isomorphic to $Heis(p)$ then the obtained group is 
generalized Heisenberg group $Heis(\F_p^n)$ and it is $p$-torsion. If at least 
one of the $n$ factors is of the form $C_{p^2} \rtimes C_p$ then we obtain 
another, non-isomorphic extra special group of order $p^{1+2n}$, but now having 
some elements of order $p^2$.

For $p=2$, a similar situation occurs: any extra special $2$-group of order at 
least $8$ is a central product of several copies of $\mathrm{Dih}_4$ and $Q_8$. If 
all of the $n$ factors appearing in the central product are isomorphic to 
$\mathrm{Dih}_4$ then the obtained group is the generalized Heisenberg group 
$Heis(\F_2^n)$. If at least one of the $n$ factors is of the form $Q_8$ then we 
obtain another, non-isomorphic extra special group of order $2^{1+2n}$.

An extra special $p$-group $G$ of order $p^{1+2n}$ has precisely $p^{2n}$ 
representations of dimension $1$ on which the center $Z(G)$ acts trivially, and 
$p-1$ representations of dimension $p^n$ on which the center $Z(G)$ acts by a 
non-trivial character.

\subsection{Preliminaries on McKay graphs}
The definitions and statements of this subsection are taken from \cite{McKay}.
\begin{definition}[McKay graph]
Let $G$ be a finite group, $\rho$ a complex representation of $G$. The 
{\it McKay graph} $\Gamma(G, \rho)$ (sometimes called ``McKay quiver'') for the 
pair $(G, \rho)$ has vertex set $Irr(G)$; the number of edges $\mu \to \lam$, 
where $\mu,\lam \in Irr(G)$, is 
$$\dim \Hom_G(\mu \otimes \rho,\lam) = [\mu\otimes \rho:\lam]_G$$ where the 
latter denotes the multiplicity of $\lam$ in the $G$-representation $\mu\otimes 
\rho$.
\end{definition}
When drawing a McKay graph, we will usually mark by a $\bigstar$ symbol the vertex corresponding to the trivial representation. The other vertices are each marked by the dimension of the corresponding irreducible representation. 

\begin{example}

\mbox{}

 \begin{enumerate}
 \item Let $G= C_2$ and $\sgn$ be the non-trivial irreducible representation 
of $C_2$ (``sign representation''). 
 Then $\Gamma(C_2, \sgn)$ is an undirected loop-less graph on $2$ vertices:
 $$
\begin{tikzcd}[arrows=-]
 \bigstar\arrow[r] & \raisebox{.5pt}{\textcircled{\raisebox{-.9pt} {1}}}
\end{tikzcd}
$$

 \item Let $G= \Z/n\Z$ and $\tau$ be its non-trivial irreducible 
representation. 
Then $\Gamma(\Z/n\Z, \tau)$ is a directed cycle graph on $n$ vertices, without 
double edges nor loops. 
 \item \InnaB{Let $n\geq 1$ and let $G = \mathbb{F}_2^n = C_2^{\times n}$ be the corresponding elementary abelian $2$-group}. Let $\rho = \bigoplus_{0 \leq s 
\leq n-1} \triv^{\boxtimes s} \boxtimes \sgn \boxtimes \triv^{\boxtimes n-s-1}
  $. The McKay graph $\Gamma(\mathbb{F}_2^n, \rho)$ is then the $n$-dimensional cube 
$\{0, 1\}^n$ with undirected edges $\xymatrix{x \ar@{-}[r]&y}$ whenever $x,y \in \{0, 1\}^n$ differ by 
precisely one coordinate.
 \item Let $G = S_3 = \mathrm{Dih}_3$ be the symmetric group on $3$ letters. We denote its 
irreducible representations by $\triv, \sgn, \mathtt{ref}$ where $\sgn$ is the 
$1$-dimensional sign representation and $\mathtt{ref}$ is the two-dimensional reflection 
representation $\{(x_1, x_2, x_3) \in \C^3, \sum x_i =0\}$. The McKay graph 
$\Gamma(S_3, \mathtt{ref})$ is 
$$ 
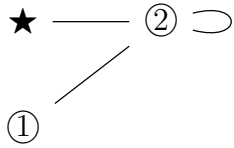
\begin{tikzcd}[arrows=-]
 \bigstar\arrow[r] & \raisebox{.5pt}{\textcircled{\raisebox{-.9pt} {2}}} \arrow[ld] \arrow[loop right]{r}\\
  \raisebox{.5pt}{\textcircled{\raisebox{-.9pt} {1}}} &{}
\end{tikzcd}
$$
 This is an undirected graph on $3$ vertices, with a single loop.
 \item Consider the dihedral group $\mathrm{Dih}_{4}$ and let $\tau$ be its tautological $2$-dimensional (irreducible) 
representation. The McKay graph of $(\mathrm{Dih}_{4}, \tau)$ looks as follows:
$$
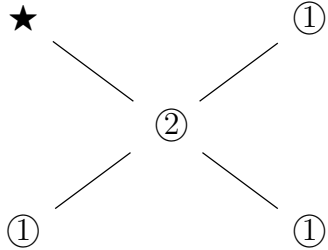
\begin{tikzcd}[arrows=-]
    \bigstar&& \raisebox{.5pt}{\textcircled{\raisebox{-.9pt} {1}}}\\
&\ \raisebox{.5pt}{\textcircled{\raisebox{-.9pt} {2}}}\arrow[ru]\arrow[lu]\arrow[ld]\arrow[rd]\\
 \raisebox{.5pt}{\textcircled{\raisebox{-.9pt} {1}}} && \raisebox{.5pt}{\textcircled{\raisebox{-.9pt} {1}}}
\end{tikzcd}
$$
Here the vertex in the center corresponds to $\tau$.
 This graph coincides with the McKay graph for $(Q_8, \mu)$ where $Q_{8}$ is the 
quaternion group and $\mu$ is the obvious $2$-dimensional (irreducible) 
representation given by $\mathbb{H} \cong \C^2$. This coincidence is not 
surprising, given that the character tables of $\mathrm{Dih}_{4}$, $Q_8$ coincide.
 \end{enumerate}
\end{example}

\begin{example}\label{ex:extra_special}
Consider an extra-special $2$-group $G$ of order $2^{1+2n}$, $n\geq 0$ as 
described in Section \ref{ssec:extra_special} (for $n=0$, the corresponding extra special group is unique, and it is just $C_2$). Then $G$ has $2^{2n}+1$ 
irreducible representations: $2^{2n}$ irreducible representations of dimension 
$1$ coming from $G/Z(G) \cong \mathbb{F}_2^{2n}$ and one irreducible representation 
$\rho$ of dimension $2^n$. The McKay graph for $(G, \rho)$ is an unoriented 
connected graph of th following form: it has one vertex $\rho$ with $\InnaA{4}^n$ 
unoriented edges coming out, and $\InnaA{4}^n$ vertices connected only to $\rho$ by a 
single unoriented edge. \InnaB{We will call such a graph {\it a hedgehog with $4^n$ spines}.} Below is an example for $n=2$: 
$$
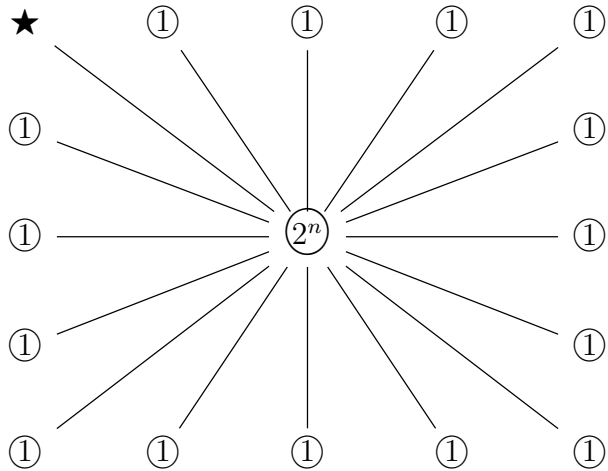
\begin{tikzcd}[arrows=-]
    \bigstar& \raisebox{.5pt}{\textcircled{\raisebox{-.9pt} {1}}}& \raisebox{.5pt}{\textcircled{\raisebox{-.9pt} {1}}}& \raisebox{.5pt}{\textcircled{\raisebox{-.9pt} {1}}}& \raisebox{.5pt}{\textcircled{\raisebox{-.9pt} {1}}}\\
     \raisebox{.5pt}{\textcircled{\raisebox{-.9pt} {1}}}&&&& \raisebox{.5pt}{\textcircled{\raisebox{-.9pt} {1}}}\\
 \raisebox{.5pt}{\textcircled{\raisebox{-.9pt} {1}}}&& {\Large\raisebox{.5pt}{\textcircled{\raisebox{-.9pt} {\normalsize $2^n$}}}}\arrow[rr]\arrow[rru]\arrow[rruu]\arrow[ruu]\arrow[uu]\arrow[ll]
\arrow[llu]\arrow[lluu]\arrow[luu]\arrow[rrd]\arrow[rrdd]\arrow[rdd]\arrow[dd]
\arrow[lld]\arrow[lldd]\arrow[ldd]&& \raisebox{.5pt}{\textcircled{\raisebox{-.9pt} {1}}}\\
     \raisebox{.5pt}{\textcircled{\raisebox{-.9pt} {1}}}&&&& \raisebox{.5pt}{\textcircled{\raisebox{-.9pt} {1}}}\\
     \raisebox{.5pt}{\textcircled{\raisebox{-.9pt} {1}}}& \raisebox{.5pt}{\textcircled{\raisebox{-.9pt} {1}}}& \raisebox{.5pt}{\textcircled{\raisebox{-.9pt} {1}}}& \raisebox{.5pt}{\textcircled{\raisebox{-.9pt} {1}}}& \raisebox{.5pt}{\textcircled{\raisebox{-.9pt} {1}}}
\end{tikzcd}
$$
\end{example}

The following statement is well-known (see \cite[Proposition 2]{McKay}):
\begin{lemma}\label{lem:undirected}
 Let $\rho$ be a representation of the group $G$. Then $\Gamma(G, \rho^*)$ is 
obtained from $\Gamma(G, \rho)$ by a reversal of arrows (i.e. $A_{\Gamma(G, 
\rho^*)} = A_{\Gamma(G, \rho)}^T$). In particular, $\rho\cong \rho^*$ if and 
only if $\Gamma(G, \rho)$ is undirected (i.e. $A_{\Gamma(G, \rho)}$ is a 
symmetric matrix).
\end{lemma}

The spectrum of the adjacency matrix of a McKay graph has an explicit description, see 
\cite[Proposition 6]{McKay}, \cite[Proposition 2.3]{Browne}, and 
\cite{Steinberg}:

\begin{proposition}\label{prop:spectra_adj_matrix}
 The eigenvalues of $A_{\Gamma(G, \rho)}$ form the multiset $$\{\chi_{\rho}(g) |g\in G//G\}.$$ A corresponding eigenvector for the 
eigenvalue $\chi_{\rho}(g) $ is $\chi(g):Irr(G) \to \C, \, \mu \mapsto 
\chi_{\mu}(g)$ (a column of the character table).
\end{proposition}

Since $\chi_{\rho}(g), g\in G$ is the sum of eigenvalues of $\rho(g)$, all of 
which are roots of unity, we have:
\begin{corollary}\label{cor:spectral_radius_adj_matrix}
 The eigenvalue of $A_{\Gamma(G, \rho)}$ with the maximal absolute value (the 
{\it spectral radius} of $A_{\Gamma(G, \rho)}$) is $\chi_\rho(1_G) = \dim \rho$. 
A corresponding eigenvector is then $$\dim:Irr(G) \to \C, \, \mu \mapsto 
\dim(\mu).$$
\end{corollary}

The spectral radius of a real matrix $A$ with non-negative entries has an important property given by the Frobenius-Perron theorem: the corresponding eigenspace of $A$ contains a vector all of whose entries are non-negative. 

In fact, if $A$ is an ``irreducible'' matrix (one which cannot be presented as a block matrix with non-trivial blocks), the spectral radius has multiplicity $1$ as an eigenvalue of $A$, and it is the only eigenvalue of $A$ for which an eigenvector with positive entries exists. For example, the adjacency matrix of a {\it strongly connected} graph $\Gamma$ (one where there is a directed walk from each vertex to any other vertex) is always irreducible, and hence its spectral radius has multiplicity $1$.

\subsection{Shapes of McKay graphs} Let us give some basic examples of graphs which can or cannot appear as McKay 
graphs $\Gamma(G, \rho)$ for some $G, \rho$.
\begin{example}
\begin{enumerate}
\item Given a group $G$ and a representation $\rho$ of $G$, the representation 
$\rho$ is $1$-dimensional if and only if the McKay graph $\Gamma(G, \rho)$ is a 
disjoint union of directed cycles (some of the cycles might consist 
of $1$ vertex only, with a single loop). 

Any disjoint union of directed cycles of equal size can appear as a McKay 
graph: 
for example, if we have $k$ disjoint directed cycles of size $n$, this is the 
McKay graph for $(\Z/n\Z \times G, \rho\boxtimes \triv)$ where $\tau$ is an 
irreducible representation of $\Z/n\Z$, and $G$ is any finite group with 
precisely $k$ irreducible representations (e.g. $G=\Z/k\Z$). 
 \item  A McKay graph $\Gamma(G, \rho)$ cannot be a directed tree or a 
disjoint union of such: indeed, one cannot have ``leaves'' (i.e. vertices with 
only one edge coming in or going out) in a directed McKay graph.
 \item Considering the opposite extreme situation, a complete (undirected, 
loop-less) graph on $n$ vertices - such a graph can be obtained as 
$\Gamma(\Z/n\Z, \rho)$. Here $\rho$ is the representation of $\Z/n\Z$ on 
$\C^n/span\{(1,\ldots, 1)\}$, where $\Z/n\Z$ acts on $\C^n$ by shifting cyclically the 
coordinates. 
 \item A complete graph on $n$ vertices ($n>2$) cannot be obtained as a McKay 
graph $\Gamma(G, \rho)$ if we require $\rho \in Irr(G)$. This is due to the 
fact that in any McKay graph $\Gamma(G, \rho)$ with $\rho \in Irr(G)$, the 
vertex $\triv$ has only one neighbor: the vertex $\rho$.
 
 Yet a given a complete (undirected, loop-less) graph on $n$ vertices, we can 
always find a McKay graph $\Gamma(G, \rho)$  containing our complete graph as a 
sub-graph: for instance, take $G = S_N$ (the symmetric group) for $N>>n$ and 
$\rho$ to be the reflection representation.
\end{enumerate}

\end{example}

\subsection{Connected components in McKay graphs}

Given two vertices $\mu, \lam$ in a McKay graph, we can ask whether there is a directed walk from $\mu$ to $\lam$, or whether there is a walk from $\mu$ to $\lam$ if we forget about the directions of the edges in our graph. Although for general graphs these relations do not coincide, it turns out that in McKay graphs they do (see \cite[Section 3]{Browne}), so the notion of a connected component in a McKay graph is intuitive; each such component is ``strongly connected''.

Given a McKay graph $\Gamma(G, \rho)$, the connected component $\Gamma_{\triv}$ containing the vertex $\triv$ is called the {\it principal component}.

The following statement is well-known, see for example \cite[Proposition 1]{McKay}, \cite[Proposition 3.3]{Browne}:
\begin{lemma}\label{lem:connectivity}
 Let $\rho$ be a representation of the group $G$. Then $\rho$ is faithful if and 
only if $\Gamma(G, \rho)$ is connected (that is, for any $\mu, \lam \in Irr(G)$ 
there exists a walk from $\mu$ to $\lam$ in $\Gamma(G, \rho)$).
\end{lemma}

In fact, given any representation $\rho$ of $G$ with kernel $N = \Ker(\rho)$, the principal component $\Gamma_{\triv}$ is isomorphic to $\Gamma(G/N, \rho)$. 
The connected components of $\Gamma(G, \rho)$ correspond to blocks in the adjacency matrix $A_{\Gamma(G, \rho)}$; the number of such blocks is then the multiplicity of the largest eigenvalue, implying the following lemma (see also \cite[Proposition 3.10, Lemma 4.1]{Browne}):
\begin{lemma}\label{lem:connected_comp_Browne} 
\mbox{}
 \begin{enumerate}
  \item The number of connected components of $\Gamma(G, \rho)$ is precisely the number of $G$-conjugacy classes of $N$; the latter equals the number of $G$-orbits in $Irr(G)$ ($G$ acts on $Irr(N)$ twisting the action of $N$ on the representations).
  
  In fact, there is a bijection
  $$\{\text{ connected components of } \Gamma(G, \rho)\} \longrightarrow Irr(N)//G$$
  given as follows: to every connected component $\Gamma'$ of $\Gamma(G, \rho)$ corresponds a unique $G$-orbit $T \subset Irr(N)$ such that $[\mu \downarrow^G_N: \tau] \neq 0$ for some $\mu \in X(\Gamma_T) \subset Irr(G)$ and some $\tau \in T \subset N$.
  \item The group $\widehat{G}$ of characters $G \to \C^{\times}$ acts on the graph $\Gamma(G, \rho)$ by tensoring each vertex with the appropriate $1$-dimensional representation.
 \end{enumerate}

\end{lemma}

\begin{corollary}\label{cor:conn_comp_isomorphic}
 Let $\Gamma'$ be a connected component of $\Gamma(G, \rho)$ with a vertex $\mu$ such that $\dim \mu=1$. Then the graph $\Gamma'$ is isomorphic to the principal component $\Gamma_{\triv}$.
\end{corollary}

In fact, we have the following useful identity:
\begin{lemma}\label{lem:conn_comp_sum_of_sq}
  Let $T \in Irr(N)//G$ and $\Gamma_T$ the corresponding connected component in $\Gamma(G, \rho)$. Let $\tau \in T$. We have:
 $$\sum_{\mu \in X(\Gamma_T)}(\dim \mu)^2 = \abs{G/N}\cdot (\dim\tau)^2 \cdot \abs{T}. $$
\end{lemma}
\begin{proof}

Given $\mu\in Irr(G)$, the multiplicity $[\mu \downarrow^G_N: \tau]$ does not depend on choice of $\tau\in T$ by Clifford's theorem, and $[\mu \downarrow^G_N: \tau']=0$ for all $\tau'\notin T$. Hence for $\mu \in X(\Gamma_T)$, we have: $ \dim \mu = \dim \tau \cdot \abs{T} \cdot [\mu \downarrow^G_N: \tau] $. In other words, $$\InnaC{[\tau\uparrow^G_N: \mu]}=[\mu \downarrow^G_N: \tau] = \frac{\dim \mu }{\dim \tau \cdot \abs{T}}.$$ 

Consider $\tau\uparrow^G_N$; by Lemma \ref{lem:connected_comp_Browne}, all its irreducible direct summands are in $X(\Gamma_T)$. We have: 
\begin{align*}
 \abs{G/N}\cdot\dim \tau  = \dim (\tau\uparrow^G_N) = \sum_{\mu \in X(\Gamma_T)}[\tau\uparrow^G_N: \mu]\dim \mu  = \sum_{\mu \in X(\Gamma_T)}\frac{(\dim \mu )^2}{\dim \tau  \cdot \abs{T}}
\end{align*}
The required statement is proved.
 
\end{proof}

\section{Auxiliary results}\label{sec:Aux}

\subsection{McKay graphs and characters}
Let $\Gamma := \Gamma(G, \rho)$ be the McKay graph for a finite group $G$ and a representation $\rho$ of $G$. 
\begin{lemma}\label{lem:trace}
 Let $k\geq 1$. We have: 
 $$
\dim \Hom_G(\rho^{\otimes k}, \C[G]^{adj}) = \sum_{g \in G//G} \chi_\rho(g)^k = tr(A^k_{\Gamma}) = \abs{\{\text{circuits of length } k \text{ in } \Gamma \}}.  $$

\end{lemma}
\begin{proof}
 The eigenvalues of the matrix $A_{\Gamma}$ are precisely $\chi_\rho(g)$ for $g \in G//G$ (see Proposition \ref{prop:spectra_adj_matrix}), so the sum of eigenvalues of $A^k_{\Gamma}$ is $\sum_{g \in G//G} \chi_\rho(g)^k $, which proves the middle equality. The right equality holds for any graph $\Gamma$ and is obvious.
 
 It remains to prove that $\dim \Hom_G(\rho^{\otimes k}, \C[G]^{adj}) = tr(A^k_{\Gamma})$. Indeed, we have:
 $$\dim \Hom_G(\rho^{\otimes k}, \C[G]^{adj}) = \dim \Hom_G(\rho^{\otimes k}, \bigoplus_{\mu \in Irr(G)} \mu^* \otimes \mu) = \sum_{\mu \in Irr(G)} \dim \Hom_G(\rho^{\otimes k}\otimes \mu, \mu)$$ 
 In the right hand side, each summand $\dim \Hom_G(\rho^{\otimes k}\otimes \mu, \mu)$ is precisely the diagonal entry in $A^k_{\Gamma}$ corresponding to $\mu \in Irr(G)$, so $$\sum_{\mu \in Irr(G)} \dim \Hom_G(\rho^{\otimes k}\otimes \mu, \mu)= tr(A^k_{\Gamma})$$ as required.
\end{proof}

\begin{corollary}\label{cor:trace}
 Assume $\Gamma$ is undirected and loop-less. Then $$\dim \Hom_G(\rho^{\otimes 2}, \C[G]^{adj}) =\sum_{g \in G//G} \chi_\rho(g)^2 = 2\abs{\{\text{undirected edges in } \Gamma\}}.$$
\end{corollary}

\subsection{McKay classification}\label{ssec:mckay_class}
In this paper, the term {\it binary polyhedral group} (denoted by appending $\mathrm{B}$ to the notation of the 
group) stands for a group $G$ which is a double cover
of the polyhedral group of given type. 

\InnaB{The polyhedral groups we consider are finite subgroups of $SO(3, \mathbb{R})$: these the rotational symmetry groups of the regular polyhedra, denoted by $\mathbf{T}$ (tetrahedral group), $\mathbf{O}$ (octahedral group) and $\mathbf{I}$ (icosahedral group), as well as the dihedral groups $\mathrm{Dih}_n$, $n\geq 2$. The latter are embedded naturally into $SO(3, \mathbb{R})$ by considering them as symmetries of ``degenerate" polyhedra: polygons which lie in the plane $XY$ of $\mathbb{R}^3$.

\begin{example}\label{ex:Dih_2}
  The subgroup of $SO(3, \mathbb{R})$ corresponding to $\mathrm{Dih}_2 \cong C_2 \times C_2$ consists of matrices $ \begin{bmatrix}
  \pm 1 &0\\
  0 &\pm 1
  \end{bmatrix}$. These are precisely the rotations around the coordinate axes by multiples of $\pi$ radians.
\end{example}}

\begin{example}
The 
{\it binary dihedral group} (also known as ``dicyclic group'') $\mathrm{B}\mathrm{Dih}_{n}$ ($n\geq 2$) has presentation
$$\langle a, x | a^{2n} = 1, x^2=a^n, x^{-1}ax = a^{-1}\rangle.$$ It has center
$Z =\{1, a^n\}\cong C_2$ and the quotient $\mathrm{B}\mathrm{Dih}_{n}/Z$ is the dihedral group 
$\mathrm{Dih}_{n}$. For $n=2$, we have $\mathrm{Dih}_{2} \cong C_2 \times C_2$ and we obtain an 
isomorphism $\mathrm{B}\mathrm{Dih}_{2} \cong Q_8$ (the quaternion group).
\end{example}

The following theorem is the most celebrated use of McKay graphs (see \cite[Proposition 4]{McKay} and \cite{Steinberg}):
\begin{theorem}[McKay's theorem]\label{thrm:mckay_class}
The following list describes all McKay graphs for pairs $(G, \rho)$ 
where $G\subset SU(2)$ is a finite group and $\rho$ is the respective
$2$-dimensional representation of $G$.

\end{theorem}

\begin{center}
\begin{tabular}{  !{\color{blue}\vrule width 2pt} c|c !{\color{blue}\vrule 
width 2pt} } 
 \specialrule{.2em}{.0em}{.0em}
 \xrowht{20pt}
 Affine Dynkin diagram & Group $G$ \\ 
 \specialrule{.2em}{.0em}{.0em} &
 \\$\tilde A_n$, $n\geq 0$  &cyclic\\ $$
\begin{tikzcd}[arrows=-]
&&\bigstar\arrow[drr]&&\\
  \raisebox{.5pt}{\textcircled{\raisebox{-.9pt} {1}}}\arrow[rru]\arrow[r] & \raisebox{.5pt}{\textcircled{\raisebox{-.9pt} {1}}}\arrow[r]
  &\cdots  \arrow[r] &
  \raisebox{.5pt}{\textcircled{\raisebox{-.9pt} {1}}} \arrow[r]  & \raisebox{.5pt}{\textcircled{\raisebox{-.9pt} {1}}}
\end{tikzcd}
$$
   &$\Z/(n+1)\Z$ \\ &  size: $n+1$\\
 \hline  &\\
 $\tilde D_n$ , $n\geq 4$ & binary dihedral  \\ $$
\begin{tikzcd}[arrows=-]
\bigstar\arrow[rd]&&&&\raisebox{.5pt}{\textcircled{\raisebox{-.9pt} {1}}}\\
   & \raisebox{.5pt}{\textcircled{\raisebox{-.9pt} {2}}}\arrow[r]
  &\cdots  \arrow[r] &
  \raisebox{.5pt}{\textcircled{\raisebox{-.9pt} {2}}} \arrow[ru]\arrow[rd] & \\
\raisebox{.5pt}{\textcircled{\raisebox{-.9pt} {1}}}\arrow[ru]&&&&\raisebox{.5pt}{\textcircled{\raisebox{-.9pt} {1}}}
\end{tikzcd}
$$
& $\mathrm{B}\mathrm{Dih}_{n-2}$ \\ & size: $4(n-2)$ \\ 
 \hline &\\
 $\tilde E_6$ & binary tetrahedral \\
 $$
\begin{tikzcd}[arrows=-]
    &&\bigstar\arrow[d]&&\\
    &&\raisebox{.5pt}{\textcircled{\raisebox{-.9pt} {2}}}\arrow[d]&&\\
  \raisebox{.5pt}{\textcircled{\raisebox{-.9pt} {1}}}\arrow[r] & \raisebox{.5pt}{\textcircled{\raisebox{-.9pt} {2}}}\arrow[r]
  &\raisebox{.5pt}{\textcircled{\raisebox{-.9pt} {3}}}  \arrow[r] &
  \raisebox{.5pt}{\textcircled{\raisebox{-.9pt} {2}}} \arrow[r]  & \raisebox{.5pt}{\textcircled{\raisebox{-.9pt} {1}}}
\end{tikzcd}
$$

  &$\mathrm{B}{\bf T}$ \\ & size: $24$ \\ 
 \hline & \\
 $\tilde E_7$ & binary octahedral\\
 $$
\begin{tikzcd}[arrows=-]
    &&&\raisebox{.5pt}{\textcircled{\raisebox{-.9pt} {2}}}\arrow[d]&&&\\
 \bigstar\arrow[r] &\raisebox{.5pt}{\textcircled{\raisebox{-.9pt} {2}}}\arrow[r] & \raisebox{.5pt}{\textcircled{\raisebox{-.9pt} {3}}}\arrow[r]
  &\raisebox{.5pt}{\textcircled{\raisebox{-.9pt} {4}}}  \arrow[r] &
  \raisebox{.5pt}{\textcircled{\raisebox{-.9pt} {3}}} \arrow[r]  &\raisebox{.5pt}{\textcircled{\raisebox{-.9pt} {2}}}\arrow[r] & \raisebox{.5pt}{\textcircled{\raisebox{-.9pt} {1}}}
\end{tikzcd}
$$

& $\mathrm{B}{\bf O}$ \\ & size: $48$   \\ 
 \hline  &\\
 $\tilde E_8$ & binary icosahedral\\
$$
\begin{tikzcd}[arrows=-]
    &&&&&
  \raisebox{.5pt}{\textcircled{\raisebox{-.9pt} {3}}}\arrow[d]&&\\
  \bigstar\arrow[r]&\raisebox{.5pt}{\textcircled{\raisebox{-.9pt} {2}}}\arrow[r]&\raisebox{.5pt}{\textcircled{\raisebox{-.9pt} {3}}}\arrow[r]&\raisebox{.5pt}{\textcircled{\raisebox{-.9pt} {4}}}\arrow[r] 
&\raisebox{.5pt}{\textcircled{\raisebox{-.9pt} {5}}}\arrow[r] & \raisebox{.5pt}{\textcircled{\raisebox{-.9pt} {6}}}\arrow[r]
  &\raisebox{.5pt}{\textcircled{\raisebox{-.9pt} {4}}}  \arrow[r]  & \raisebox{.5pt}{\textcircled{\raisebox{-.9pt} {2}}}
\end{tikzcd}
$$
& $\mathrm{B}{\bf I}$ \\ & size: $120$ \\ 
 \specialrule{.2em}{.0em}{.0em}
\end{tabular}
\end{center}

The proof of Theorem \ref{thrm:mckay_class} relies on the classification of all finite undirected connected simply-laced graphs $\Gamma$ with spectral radius $2$. The classification of such graphs states that they must be affine Dynkin diagrams of types $A$, $D$ or $E$, see \cite[Proposition 4]{McKay}.

\begin{proposition}\label{prop:McKay_classification}
  Let $G$ be a finite group and $\rho$ be its self-dual faithful 
representation (so $\Gamma(G, \rho)$ is connected and undirected). Assume 
$\dim \rho = 2$. 

Then either 
$G\subset SU(2)$ (hence $G$ appears in the list of Theorem \ref{thrm:mckay_class} and $\Gamma(G, \rho)$ is of the form above), 
or $G=\mathrm{Dih}_n$ and $\rho$ is the 
complexification of the tautological $2$-dimensional representation of 
$\mathrm{Dih}_n$. In the latter case, $\Gamma(G, \rho)$ is drawn below.
\end{proposition}

\begin{center}
\begin{tabular}{  !{\color{blue}\vrule width 2pt} c|c|c !{\color{blue}\vrule 
width 2pt} } 
 \specialrule{.2em}{.0em}{.0em}
 \xrowht{20pt}
 parity of $n$ & McKay graph for $\mathrm{Dih}_n$ and tautological $\rho$, $\dim(\rho)=2$ & number of vertices\\ 
 \specialrule{.2em}{.0em}{.0em}
 $n$ even  & $$
\begin{tikzcd}[arrows=-]
\bigstar\arrow[rd]&&&&\raisebox{.5pt}{\textcircled{\raisebox{-.9pt} {1}}}\\
   & \raisebox{.5pt}{\textcircled{\raisebox{-.9pt} {2}}}\arrow[r]
  &\cdots  \arrow[r] &
  \raisebox{.5pt}{\textcircled{\raisebox{-.9pt} {2}}}\arrow[ru]\arrow[rd] & \\
\raisebox{.5pt}{\textcircled{\raisebox{-.9pt} {1}}}\arrow[ru]&&&&\raisebox{.5pt}{\textcircled{\raisebox{-.9pt} {1}}}
\end{tikzcd}
$$ &$\frac{n}{2}+3$
 \\
 \hline
 $n$ odd & $$
\begin{tikzcd}[arrows=-]
\bigstar\arrow[rd]&&&&{}\\
   & \raisebox{.5pt}{\textcircled{\raisebox{-.9pt} {2}}}\arrow[r]
  &\cdots  \arrow[r] &
  \raisebox{.5pt}{\textcircled{\raisebox{-.9pt} {2}}} \arrow[loop right]{r} & \\
\raisebox{.5pt}{\textcircled{\raisebox{-.9pt} {1}}}\arrow[ru]&&&&{}
\end{tikzcd}
$$ &$\frac{n+3}{2}$\\ 
 \specialrule{.2em}{.0em}{.0em}
\end{tabular}
\end{center}

\begin{proof}
Since $\rho$ is 
self dual, it either has a symplectic $G$-invariant form (i.e. $\rho$ is of 
quaternionic type) or a symmetric $G$-invariant form (i.e. $\rho$ is of 
real type). In the former case, $G 
\subset SL(2)\cap U(2) =SU(2)$ and we are in the situation of Theorem 
\ref{thrm:mckay_class}. In the latter case, the representation $(\rho, V)$ is 
of real type, the corresponding $2$-dimensional real representation 
$(\rho_{\mathbb R}, V_{\mathbb R})$ satisfies: $\rho \cong 
\C\otimes_{\mathbb R}\rho_{\mathbb R}$, where $\rho_{\mathbb R}$ is a faithful $2$-dimensional 
representation of $G$ over $\mathbb R$. Hence $G\subset {\rm Dih}_n$ for some $n\geq 1$. But in that case $G$ it either cyclic (and then we are again in the situation of Theorem 
\ref{thrm:mckay_class}) or $G$ is dihedral. This 
completes the proof of the proposition.
\end{proof}

\subsection{Bipartite McKay graphs}
\begin{lemma}\label{lem:bipartite_eigen}
 Let $\Gamma$ be a bipartite graph with adjacency matrix $A_{\Gamma}$. Then for each eigenvalue $\lambda$ of $A_{\Gamma}$, $-\lambda$ is also an eigenvalue of $A_{\Gamma}$.
\end{lemma}
\begin{proof}
 Let $X(\Gamma) = X_0 \sqcup X_1$ be the bipartition of the set of vertices, such that $$\forall i\in \{1,2\}, \; \forall x, y \in X_i, \; x\notin \mathtt{N}(y).$$ Let $f: X(\Gamma) \to \C$ be the eigenvector of $A_{\Gamma}$ with eigenvalue $\lambda$. Then $$\forall x\in X(\Gamma), \; \lambda f(x) = \sum_{y \in \mathtt{N}(x)} f(y).$$ Now define a new function $\tilde{f}: X(\Gamma) \to \C$ by setting $$\forall i\in \{1,2\}, \; \forall x\in X_i, \; \tilde{f}(x) := (-1)^i f(x).$$ Then for any $x\in X_i$, we have: 
 $$-\lambda \tilde{f}(x) = -\lambda (-1)^i f(x) = \sum_{y \in \mathtt{N}(x)}(-1)^{i+1} f(y) =  \sum_{y \in \mathtt{N}(x)} \tilde{f}(y),$$ making $\tilde{f}$ an eigenvector of $A_{\Gamma}$ with eigenvalue $-\lambda$.
\end{proof}

Let $\Gamma := \Gamma(G, \rho)$ be the McKay graph for a finite group $G$ and a representation $\rho$ of $G$.
\begin{lemma}\label{lem:bipartite_center}

 Assume $\rho$ is irreducible and faithful. Then the graph $\Gamma= \Gamma(G, \rho)$ is bipartite iff $C_2 \subset Z(G)$ and this subgroup acts non-trivially on $\rho$. 
\end{lemma}

\begin{proof}

$[\mathbf{\Leftarrow}]$: Assume $C_2 \subset Z(G)$, and let $z\in C_2\subset Z(G)$, $z\neq 1_G $. Then the vertices in $\Gamma$ can be partitioned into two disjoint subsets, according to the eigenvalue by which $z$ acts on the given irreducible representation of $G$. Since $\rho$ is faithful and irreducible, $\rho(z) =-\id$. So clearly, there are no edges in $\Gamma$ between vertices $\lambda, \mu \in Irr(G)$ such that $\lambda(z) = \mu(z)$. Hence $\Gamma$ is bipartite.

 $[\mathbf{\Rightarrow}]$: Assume $\Gamma$ is bipartite. By Lemma \ref{lem:bipartite_eigen}, for each eigenvalue $\lambda$ of $A_{\Gamma}$, we also have an eigenvalue $-\lambda$.

 Next, the eigenvalues of $A_{\Gamma}$ are precisely $\chi_\rho(g)$ by Proposition \ref{prop:spectra_adj_matrix}, so there exists $z \in G$ with 
 $$\chi_\rho(z) = -\chi_{\rho}(1_G) = - \dim \rho$$
 Since $z \in G$ has finite order, $\rho(z)$ is a diagonalizable operator whose eigenvalues are all roots of unity. We have $tr(\rho(z)) = - \dim \rho$ and hence $\rho(z) = -\id$. Faithfulness of $\rho$ now implies that $z \in Z(G)$ and $z^2 = 1_G$. Hence $C_2 \subset Z(G)$ as required.
 
\end{proof}
\begin{lemma}\label{lem:faithful_center}
 Assume $\rho$ is irreducible and faithful, and that $\rho^* \cong \rho$. Then $Z(G) \subset C_2 $.
\end{lemma}

\begin{proof}
  By Schur's lemma, $\rho$ defines a character of the center $$\lambda:Z(G) \to \C^{\times}, \;\, z\mapsto \lambda(z),$$ where $\rho(z) = \lambda(z)\id$. Since $\rho^* \cong \rho$, we have: $\triv \subset \rho \otimes \rho^* \cong \rho^{\otimes 2}$ and on this space $z \in Z(G)$ acts by $\lambda^2(z)\id$. So $\lambda(z)^2 = 1$ and thus $\lambda(z)=\pm 1$. Now, since $\rho$ is faithful, the map $\lambda: Z(G) \to \C$ is injective, hence $Z(G) \subset C_2 $.
\end{proof}

Recall that by Lemma \ref{lem:undirected}, $\rho^* \cong \rho$ iff $\Gamma$ is undirected, and by Lemma \ref{lem:connectivity}, $\rho$ is faithful iff $\Gamma$ is connected. Using Lemmas \ref{lem:bipartite_center}, \ref{lem:faithful_center}, we conclude:

\begin{corollary}\label{cor:bipartite_center_Z_2}
 Assume $\rho$ is irreducible and $\Gamma = \Gamma(G, \rho)$ is an undirected and connected graph. Then $\Gamma$ is bipartite iff $Z(G) = C_2 $.
\end{corollary}

\section{Trees}\label{sec:trees}
\begin{theorem}\label{thrm:main}
 Let $\Gamma := \Gamma(G, \rho)$ be the McKay graph for a finite group $G$ and a representation $\rho$ of $G$. 
 
 Assume $\Gamma$ is an (undirected) tree. 
 
 Then $\rho$ is a faithful irreducible representation of $G$, and one of the following conditions is satisfied:
 \begin{itemize}
 \item $\dim \rho =2$, and the graph $\Gamma$ is an affine Dynkin diagram as in Proposition \ref{prop:McKay_classification}; in that case $(G, \rho)$ belongs to the list appearing in Theorem \ref{thrm:mckay_class}, or $G=Dih_{2n}$ and $\rho$ is its natural $2$-dimensional representation.
  \item $G$ is an extra special $2$-group of order  $2^{1+2n}$ ($n\geq 0$), with $\rho$ its unique irreducible $2^n$-dimensional representation on which the center of $G$ acts non-trivially; in that case $\Gamma$ is a ``hedgehog'' as in Example \ref{ex:extra_special}, whose number of spines (leaves) is $4^n$. 
 \end{itemize}

\end{theorem}

\begin{proof}
First of all, since $\Gamma$ is undirected and connected, we have, by Lemmas \ref{lem:connectivity}, \ref{lem:undirected}, that $\rho$ is a self-dual, faithful representation of $G$. Next, the condition that $\Gamma$ is a tree implies that $$\abs{\{\text{undirected edges in } \Gamma\}} = 2(\abs{Irr(G)}-1) = 2(\abs{G//G} -1).$$

By Corollary \ref{cor:trace} we have:
 \begin{equation}\label{eq:sq_char}
\dim \Hom_G(\rho^{\otimes 2}, \C[G]^{adj}) = tr(A_{\Gamma}^{\otimes 2})= 2(\abs{G//G} -1).  
 \end{equation}

 Now, 
 $$\C[G]^{adj} = \bigoplus_{[g] \in G//G} span\{g |g\in [g]\} \cong \bigoplus_{[g] \in G//G} \C[G/C(g)] \cong \bigoplus_{[g] \in G//G} \C\uparrow^G_{C(g)}, $$
 where $C(g)$ is the centralizer of an element $g$ in the conjugacy class $[g]$ (the isomorphisms depend on the choices of representatives $g\in [g]$, of course).

 Thus we have:
\begin{align}\label{eq:sum_dim_end}
 \nonumber&2(\abs{G//G} -1) = \dim \Hom_G(\rho^{\otimes 2}, \C[G]^{adj})  = \sum_{[g] \in G//G} \dim \Hom_G(\rho^{\otimes 2}, \C\uparrow^G_{C(g)} )=  \\ &\sum_{[g] \in G//G} \dim \Hom_{C(g)}\left(\rho^{\otimes 2} \downarrow_{C(g)}, \triv\right)= \sum_{[g] \in G//G}  \dim \End_{C(g)}(\rho \downarrow_{C(g)})
\end{align}
(the last equality follows from the fact that $\rho$ is self-dual).

Let us compute $ \dim \End_{C(g)}(\rho \downarrow_{C(g)})$. This value is clearly at least $1$, and we will show that this is precisely $2$ when $g\notin Z(G)$.

Assume $\rho$ is reducible as a $G$-representation. Hence it is reducible as a $C(g)$-representation for each $g\in G$. This implies: $$\dim \End_{C(g)}(\rho \downarrow_{C(g)})\geq 2$$ for all $g\in G$ and so $$\sum_{[g] \in G//G}  \dim \End_{C(g)}(\rho \downarrow_{C(g)}) \geq 2\abs{G//G} >2(\abs{G//G} -1)$$ which is a contradiction. Hence $\rho$ is irreducible.  

For $g\in Z(G)$ we have: $C(g)=G$ so clearly $ \dim \End_{C(g)}(\rho \downarrow_{C(g)})=1$ in that case.

Now let $g\in G \setminus Z(G)$. The operator $\rho(g)$ is an intertwining operator on the $C(g)$-representation $\rho \downarrow_{C(g)}$ and hence acts by scalar on any simple $C(g)$-summand of $\rho \downarrow_{C(g)}$. Note that $\rho(g)$ cannot be a scalar endomorphism: indeed, if that were the case, $\rho(g)$ would commute with $\rho(h)$ for any $g\in G$. Since $\rho$ is faithful, this would imply that $g \in Z(G)$, contrary to our assumption on $g$.

Hence $\rho(g)$ is not a scalar endomorphism, and so $\rho \downarrow_{C(g)}$ is not simple.
 
Hence we have: for any $g\in G \setminus Z(G)$, $\dim \End_{C(g)}(\rho \downarrow_{C(g)}) \geq 2$.

 Next, by Corollary \ref{cor:bipartite_center_Z_2} we have: $Z(G) = C_2$. So by Equation \eqref{eq:sum_dim_end}, we have:
 $$\sum_{\substack{[g] \in G//G,\; g \notin Z(G)}} \dim \End_{C(g)}(\rho \downarrow_{C(g)}) = 2(\abs{G//G} -2).$$
 
 Notice that the left hand side has $(\abs{G//G} -2)$ summands, so the average value of $\dim \End_{C(g)}(\rho \downarrow_{C(g)}) $ is $2$. But since we showed that the minimal value is also $2$, we conclude that $\dim \End_{C(g)}(\rho \downarrow_{C(g)}) =2$ for any $g\notin Z(G)$, i.e. the representation $\rho \downarrow_{C(g)}$ has length $2$.
 
 If $C(g)$ is abelian for some $g\notin Z(G)$, then $\rho \downarrow_{C(g)}$ is a direct sum of two $1$-dimensional representations and the statement of the theorem is proved. If $G$ itself is abelian, then $\dim \rho = 1$ and $\Gamma$ is tree only if $G= C_2$, $\rho = \sgn$; again, the statement of the theorem is then proved. 
 
Thus we will assume from now on that $C(g)$ is not abelian for any $g\in G$.
 
Since $\dim \End_{C(g)}(\rho \downarrow_{C(g)}) =2$, for any $g\notin Z(G)$, $\rho(g)$ is a diagonalizable endomorphism with precisely $2$ distinct eigenvalues. Furthermore, since $\rho \cong \rho^*$ as $G$-representations, we have: for each eigenvalue $z$ of $\rho(g)$, its complex conjugate $\bar{z}$ is also an eigenvalue of $\rho(g)$ with the same multiplicity. 

So the set of eigenvalues of each $\rho(g), g\notin Z(G)$ is either $\{\pm 1\}$ or $\{z, \bar{z}\}$ where $z$ is a root of unity of order $\abs{G}$, $z\neq \pm 1$. Note that the second option can occur only if $\dim \rho$ is even.

Assume the eigenvalues of any $\rho(g), g\in G$ belong to $\{\pm 1\}$. This implies that $G$ is a $2$-torsion group, meaning that each element in $G$ has order $2$. In particular, $$\forall g, h\in G, \, ghg^{-1}h^{-1} = ghgh=1_G \; \Rightarrow \; gh=hg$$ so $G$ is abelian. But that contradicts our assumption. 


So there exists $g \in G$ such that $g^2 \neq 1_G$, i.e. the eigenvalues of $\rho(g)$ are $z_1 \neq z_2$ (these are roots of unity, and $z_1 =\bar{z_2}$). Let $V_{z_k}$ denote the eigenspace of $\rho(g)$ corresponding to the eigenvalue $z_k$, $k=1,2$; then $V_{z_1} \oplus V_{z_2}$ is the entire space underlying the representation $\rho$, and $\dim V_{z_1} = \dim V_{z_2}$. 

Let $h\in C(g)$. The operators $\rho(h), \rho(g)$ commute so they are simultaneously diagonalizable ($\rho(h)$ also having at most $2$ distinct eigenvalues). We will say that $h$ is {\it aligned} with $g$ if the restriction of $\rho(h)$ to each of the subspaces $V_{z_k}, k=1,2$ is a scalar endomorphism. The set of elements $h\in C(g)$ aligned with $g$ clearly forms an abelian subgroup.

Since we assumed that the group $C(g)$ is itself not abelian, we can find $h \in C(g)$ which is not aligned with $g$. We have: $h\notin Z(G)$, so we denote by $t_1 \neq t_2$ the distinct eigenvalues of $\rho(h)$. Since $h$ is not aligned with $g$, the eigenvalues of $\rho(gh)$ would be $\{z_k t_j |k, j\in \{1,2\}\}$ or at least $3$ of these products. This implies that in at least one of the pairs $(t_1 z_1, t_2 z_2)$, $(t_1 z_2, t_2 z_1)$ the elements coincide, implying that $z_1^2 =\pm 1$ so $z_1, z_2 \in \{\pm i\}$. So we see that for every $g\in G$, the eigenvalues of $\rho(g)$ are either $\{\pm i\}$ or a subset of $\{\pm 1\}$.

In particular, $\rho(g)^2 = \pm \id$ for each $g\in G$, so $G/Z(G)$ is a $2$-torsion group and hence of the form $\mathbb{F}_2^m$ (an elementary abelian $2$-group). Since $Z(G)=C_2$, we conclude that $G$ is an extra special $2$-group, $m$ is even and $\rho$ is the unique irreducible representation of $G$ of dimension greater than $1$ (see Example \ref{ex:extra_special}). This concludes the proof of the theorem.

\end{proof}

\section{Forests}\label{sec:forests}
In this section we investigate the case of an undirected forest graph, i.e. a disjoint union of undirected trees.

\subsection{Classification of McKay forests}

\begin{proposition}\label{prop:forest}
 Let $\Gamma := \Gamma(G, \rho)$ be the McKay graph for a finite group $G$ and an representation $\rho$ of $G$. 
 
 Assume $\Gamma$ is an (undirected) forest. Then $\rho$ is a irreducible, self-dual representation.
\end{proposition}

\begin{proof}
Let $N \InnaB{:=} Ker(\rho) \vartriangleleft G$ and let $K$ be the number of connected components of $\Gamma$. By Lemma \ref{lem:connected_comp_Browne}, we have: $K$ is the number of $G$-conjugacy classes in $N$. 

The principal component $\Gamma_{\triv} = \Gamma(G/N, \rho)$ is a tree, so Theorem \ref{thrm:main} shows that $\rho$ is a self-dual irreducible representation of $G/N$. Hence $\rho$ possesses the same properties when considered as a $G$-representation, which proves the first part of the statement. 
\end{proof}

\begin{theorem}\label{thrm:forest}
  Let $\Gamma := \Gamma(G, \rho)$ be the McKay graph for a finite group $G$ and a representation $\rho$ of $G$. Let $N:=Ker(\rho)$.
 
 Assume $\Gamma$ is an undirected forest. Then we have:
 \begin{enumerate}
  \item All the connected components of $\Gamma$ are of one of the forms described in Theorem \ref{thrm:main}.
  \item If one of the components is a hedgehog with $4^n$ spines for $\InnaB{n\neq 1}$, then all the components of $\Gamma$ are of isomorphic to each other (\InnaB{in particular}, they are isomorphic hedgehogs).
 \item \InnaB{Assume} one of the components $\Gamma_T$ ($T\in Irr(N)//G$) is an affine Dynkin diagram of type $D$ or $E$. \InnaB{Let} $G'_T$ be the dihedral or binary polyhedral group corresponding to this component by McKay's classification (Theorem \ref{thrm:mckay_class}). 
 
 Then $|G/N|$ is divisible by $|G'_T|$. 
 \end{enumerate}

\end{theorem}
\begin{remark}
 For $\InnaB{n=1}$, the hedgehog with $4$ spines is the  affine Dynkin diagram $\widetilde{D}_4$.
\end{remark}

\begin{remark}
 In particular, if one of the components is an affine Dynkin diagram of types $D$ of $E$, the remaining components are also affine Dynkin diagrams of types $D$ or $E$.
 
 They do not have to be isomorphic to each other, as we will show in \InnaC{the examples of Section \ref{ssec:forest_examples}}.
 
 However, the theorem above states that the size of the dihedral or binary polyhedral group corresponding to the principal component (that is, $G/N$) will be divisible by the sizes of the dihedral or binary polyhedral groups corresponding to other connected components. 
 
 For instance, if the principal component is of type $\widetilde{E}_6$, then the only possibility for other connected components is to be of types $\widetilde{E}_6$ or $\widetilde{D}_4$.
  
\end{remark}
\InnaB{
For the proof, we need the following standard lemma:
\begin{lemma}\label{lem:graph_aux}
Given a tree $\Gamma$, if there exists a unique vertex $\nu$ connected to a leaf, then $\Gamma$ is a hedgehog.  
\end{lemma}
}
 \begin{proof}[Proof of Theorem \ref{thrm:forest}]
The principal component $\Gamma_{\triv}$ of $\Gamma$ is a tree, and a McKay graph in its own right: $\Gamma_{\triv} = \Gamma(G/N, \rho)$. Hence it is isomorphic one of the graphs given by Theorem \ref{thrm:main}. 

Consider a connected component $\Gamma'$ of $\Gamma$ corresponding to some $T\in Irr(N)//G$. Let $s$ be the dimension of any representative of $T$. By Lemma \ref{lem:conn_comp_sum_of_sq} we have:
\begin{equation}\label{eq:sum_squares_conn_comp}
 |T|s^2\abs{G/N} = \sum_{\mu \in X(\Gamma')}(\dim \mu)^2.
\end{equation}
 
We will now consider the different possibilities for $\Gamma_{\triv}$.

{\bf Case when $\Gamma_{\triv}$ is a ``hedgehog'' with $\InnaB{4}^n$ spines, $n\geq 0$}: in that case, $\dim \rho = 2^n$ \InnaC{and $|G/N|=2^{2n+1}$. The tree $\Gamma'$ is not a singleton, so} for each leaf $\tau$ in $\Gamma'$, we have a single vertex $\mu$ connected to it. Hence $$ \dim\mu = \dim \rho \cdot \dim\tau = 2^n \cdot \dim\tau.$$ Let $a$ be the multiplicity of any representative of $T$ in $\tau \downarrow_{N}$: then $\tau \downarrow_{N} = \bigoplus_{\sigma\in T} \sigma^{\oplus a}$ so $\dim\tau = a|T|s$. Hence 
\begin{align*}
(\dim \tau)^2 + (\dim\mu)^2 &= (\dim \tau)^2 (1+2^{2n}) = |T|^2 s^2 a^2(1+2^{2n})\leq |T|s^2\abs{G/N} \\ &= |T|s^2 2^{1+2n} = |T|s^2 2\cdot 2^{2n}. 
\end{align*}
\InnaB{
Thus $a^2|T| \InnaB{\leq \frac{2\cdot 2^{2n}}{1+2^{2n}}} <2$ and so $a=|T|=1$, which implies $\dim\tau = s$, \InnaC{$\dim \mu=2^n$}. So we have:
\begin{enumerate}[label=(\roman*)]
    \item\label{itm:conn_comp_1} All leaves in $\Gamma'$ have dimension $s$.
    \item\label{itm:conn_comp_2} \InnaB{Any vertex $ \nu$ which is connected to a leaf has dimension $2^n s$.}
\end{enumerate} } 
\InnaB{This, together with
Equation \eqref{eq:sum_squares_conn_comp}
implies that either there a unique such vertex $\nu$, or $\Gamma'$ is the tree with $2$ vertices. In the former case,} all the leaves in $\Gamma'$ are connected to this vertex $\nu$. Hence \InnaB{by Lemma \ref{lem:graph_aux}, this implies that $\Gamma'$ is a hedgehog. It now follows from \eqref{eq:sum_squares_conn_comp} and from \eqref{itm:conn_comp_1}, \eqref{itm:conn_comp_2} that the hedgehog has precisely $4^n$ spines, and it is isomorphic to the principal component $\Gamma_{\triv}$.} 

This shows that when $G/N$ is an extra special $2$-group, all the connected components of $\Gamma$ are isomorphic to each other.

{\bf Case when $\Gamma_{\triv}$ is an affine Dynkin diagram of types $\tilde{D}, \tilde{E}$}: in that case $\dim \rho = 2$. \InnaB{Recall that $2$ is an eigenvalue of $A_{\Gamma}$ and there exists a corresponding eigenvector with positive integer entries; this implies that} each connected component $\Gamma'$ of $\Gamma$ is a tree with the following property: $2$ is an eigenvalue of $A_{\Gamma'}$ and there exists a corresponding eigenvector with positive integer entries. Hence the spectral radius of $\Gamma'$ is $2$ and so by McKay's classification (\cite[Proposition 4]{McKay}) described in Section \ref{ssec:mckay_class} we have: each connected component of $\Gamma$ is isomorphic to an affine Dynkin diagram of types $\tilde{D}$ or $\tilde{E}$. In particular, it cannot be a hedgehog with $\InnaB{4}^n$ spines for $n\neq 1$. 

\InnaA{Let $\Gamma'=\Gamma_T$ be a connected component of $\Gamma$ corresponding to $T \in Irr(G)//N$. Let $G'_T$ be the dihedral or binary polyhedral group corresponding to this component by McKay's classification (Theorem \ref{thrm:mckay_class}). \InnaB{Let us choose an identification of $\Gamma'$ with the McKay graph of $G'_T$ corresponding to its natural $2$-dimensional representation.}

Let $d: X(\Gamma_T) \to \Z$ denote the natural markings on the vertices of the affine Dynkin diagram $\Gamma_T$: that is, $\dim(\tau)$ for $\tau \in Irr(G'_T)$. Then $$\sum_{\mu \in X(\Gamma_T)}d(\mu)^2 =\abs{G'_T}.$$

\InnaB{Recall that by the Frobenius-Perron theorem, an irreducible square matrix $A$ with non-negative entries has a unique (up to a real positive scalar) eigenvector with positive entries, and this eigenvector corresponds to the maximal eigenvalue of $A$.} Thus the eigenvector $\dim: X(\Gamma_T) \to \Z $ of $A_{\Gamma_T}$ corresponding to the dimensions of vertices $\mu \in \Gamma_T$ is a \InnaB{positive} multiple of the eigenvector $d: X(\Gamma_T) \to \Z$ of $A_{\Gamma_T}$. Let $a \in \mathbb{R}_{>0}$ be such that $\dim = a\cdot d$. 
Then by Equation \eqref{eq:sum_squares_conn_comp}, \begin{equation}\label{eq:conn_comp_aux2}
 |T|s^2\abs{G/N}= \sum_{\mu \in X(\Gamma_T)} (\dim\mu)^2 = \sum_{\mu \in X(\Gamma_T)}a^2 d(\mu)^2 = a^2 \abs{G'_T}  
\end{equation}
(here, as before, $s$ stands for the dimension of any representative of $T$).

Now, denote by $\tau_0 \in Irr(G)$ the leaf of $\Gamma_T$ which corresponds to the trivial representation of $G'_T$. Then $d(\tau_0)=1$ so $a = \dim(\tau_0)$. \InnaB{Thus  $a \in \Z_{> 0}$.}

Consider the restriction $\tau_0 \downarrow_N$; we have 
$\tau_0 \downarrow_N = \left(\bigoplus_{\sigma \in T} \sigma\right)^{\oplus b}$ for some $b> 0$.

Hence $a = \dim(\tau_0) = |T|sb$. \InnaB{Substituting this into \eqref{eq:conn_comp_aux2}, we obtain: }
$$|T|s^2 |G/N| = a^2|G'_T| = |T|^2 s^2 b^2 |G'_T|$$
which implies $|G/N| = |T| b^2 |G'_T|$. Thus $|G/N|$ is divisible by $|G'_T|$.

The proof of the theorem is complete.}

\end{proof}
\subsection{Examples}\label{ssec:forest_examples}

\subsubsection{Construction of unions of graphs}\label{sssec:constr_forests}

 \InnaB{Let $G$ be a group with a normal subgroup $H$ and a fixed isomorphism $G/H \cong \mathbb{F}_p^{\times}$, where $p$ is prime. 
 
 For any representation $\rho$ of $G$, we will construct a McKay graph which is the disjoint union of graphs $\Gamma(G, \rho)$ and $\Gamma(H, \rho\downarrow_H)$.
 
 Let $\InnaB{C_p}$ be a cyclic group on $p$ elements. Then \InnaB{our chosen isomorphism} defines a homomorphism $G/H\to Aut(\InnaB{C_p}) \cong \mathbb{F}_p^{\times}$; in other words, it defines an action of $G$ on $\InnaB{C_p}$ by group automorphisms. The set $Irr(\InnaB{C_p})//G$ has precisely two elements, which are $\{\triv_{\InnaB{C_p}}\}$ and the orbit of all the non-trivial group characters of $\InnaB{C_p}$.
 
 Fix $\xi \in Irr(\InnaB{C_p})$ a non-trivial character. \InnaB{We will consider it as a character of $H\times \InnaB{C_p}$.}
 
 Let $G':=G\ltimes \InnaB{C_p}$ be the semidirect product. Then by \InnaB{the classification of irreducible representations of semidirect products with an abelian subgroup (see e.g. \cite[Section 8.2]{Serre})}, we have a bijection $$S:Irr(G) \sqcup Irr(H)\stackrel{\sim}{\longrightarrow} Irr(G').$$ On $Irr(G)$, this map is the pullback with respect to the quotient $q:G'/\InnaB{C_p} \to G$, and on $Irr(H)$, this map is the induction $\tau \mapsto (\tau\otimes \xi)\uparrow_{H\InnaB{\times} \InnaB{C_p}}^{G'}$.
 
Let $\rho$ be a representation of $G$. Then 
\begin{align*}
    \forall \lam, \mu \in Irr(G), \; &\text{ we have: }\;\; [S(\lam) \otimes S(\rho):S(\mu)]_{G'} = [\lam\otimes\rho :\mu]_{G},\\
    \forall \lam\in Irr(G), \mu\in  Irr(H), \;&\text{ we have: } \;\;[S(\lam) \otimes S(\rho):S(\mu)]_{G'}=0,\\
     \forall \lam\in Irr(G), \mu\in  Irr(H), \;&\text{ we have: } \;\;[S(\mu) \otimes S(\rho):S(\lam)]_{G'}=0.
\end{align*} 
\comment{since $\InnaB{C_p} \subset Ker(S(\rho))$ and $\InnaB{C_p} \subset \InnaB{Ker}(S(\lam))$, but $\InnaB{C_p} \not\subset Ker(S(\mu))$.}

Thus the vertices of $\Gamma(G', \rho)$ \InnaB{coming from} $Irr(G)$ do not lie in the same connected components as the vertices coming from $Irr(H)$.

Furthermore, \InnaB{the full subgraph on the vertices coming from $Irr(G)$ is isomorphic to $\Gamma(G, \rho)$.}

\InnaB{We will now show that the remaining components of $\Gamma(G', S(\rho))$ form the McKay graph $\InnaC{\Gamma(H, \rho\downarrow_{H})}$; that is, for $\lam, \mu\in Irr(H)$, we will prove: \begin{align*}
[S(\lam) \otimes S(\rho):S(\mu)]_{G'} &\,= \, \dim\Hom_{H}( \lam\otimes \rho\downarrow_{H},\mu) 
\end{align*}}

Indeed, for $\lam, \mu\in Irr(H)$, we have:
\begin{align*}
[S(\lam) \otimes S(\rho):S(\mu)]_{G'} &\,=\,\dim\Hom_{G'}(S(\lam) \otimes S(\rho),S(\mu)) \\
&\,=\,\dim\Hom_{G'}((\lam\otimes \xi)\uparrow_{H\InnaB{\times} \InnaB{C_p}}^{G'} \otimes S(\rho), (\mu\otimes \xi)\uparrow_{H\InnaB{\times} \InnaB{C_p}}^{G'}) \\
&\,=\,\dim\Hom_{G'}( \left(\lam\otimes \xi \otimes S(\rho)\downarrow_{H\InnaB{\times} \InnaB{C_p}}\right)\uparrow_{H\InnaB{\times} \InnaB{C_p}}^{G'},(\mu\otimes \xi)\uparrow_{H\InnaB{\times} \InnaB{C_p}}^{G'}) \\
&\,=\,\dim\Hom_{H\InnaB{\times} \InnaB{C_p}}( \lam\otimes \xi \otimes S(\rho)\downarrow_{H\InnaB{\times} \InnaB{C_p}},\bigoplus_{\bar{g}\in G/H}\mu^g\otimes \xi^g).\\
\end{align*}

Here $\mu^g$ denotes the conjugation of the $H$-representation $\mu$ by a representative $g\in \bar{g}$, and similarly $\xi^g$ \InnaC{denotes the conjugation of $\xi$ by a representative $g\in \bar{g}$}. Since $\xi^g = \xi$ precisely for $g\in H$ \InnaC{and $ S(\rho)\downarrow_{H \times C_p}$ is trivial on $C_p$}, we have: 
\begin{align*}
[S(\lam) \otimes S(\rho):S(\mu)]_{G'} &\,= \,\dim\Hom_{H\InnaB{\times} \InnaB{C_p}}(\lam\otimes \xi \otimes  S(\rho)\downarrow_{H\InnaB{\times} \InnaB{C_p}},\bigoplus_{\bar{g}\in G/H}\mu^g\otimes \xi^g)\\&=\,\dim\Hom_{H\InnaB{\times} \InnaB{C_p}}( \lam\otimes \xi \otimes S(\rho)\downarrow_{H\InnaB{\times} \InnaB{C_p}},\mu\otimes \xi) = \dim\Hom_{H}( \lam\otimes \rho\downarrow_{H},\mu) 
\end{align*}

To summarize, for any representation $\rho$ of $G$,
$$\Gamma(G', S(\rho))\; \InnaB{\cong} \; \Gamma(G, \rho) \bigsqcup \Gamma(H, \rho\downarrow_H).$$

\begin{remark}\label{rmk:constr_forests}
  The construction in Section \ref{sssec:constr_forests} also works when the isomorphism $G/H \cong \mathbb{F}_p^{\times}$ is replaced by a homomorphism
 $G/H \hookrightarrow GL_n(\mathbb{F}_p)$, where the action of $G/H$ on $\mathbb{F}_p^n\setminus \{0\}$ is transitive. The cyclic group $\InnaB{C_p}$ is then replaced by the elementary abelian $p$-group $\mathbb{F}_p^n$.
\end{remark}

}
\subsubsection{Examples of forests with non-isomorphic connected components}
\InnaB{
We will now apply the constructions of Section \ref{sssec:constr_forests} to construct examples where both $G, H$ are either binary polyhedral or dihedral groups. We will consider the tautological faithful $2$-dimensional irreducible representation $\rho$ of $G\subset SL_2(\mathbb{C})$. \InnaC{So in our examples, both} $\Gamma(G, \rho)$ and $\Gamma(H, \rho)$ \InnaC{will be} affine Dynkin diagrams of types $D$ or $E$.

\begin{example}\label{ex:forest_dihedral}
 \InnaB{Let $n>1$.} Let $\mathrm{Dih}_{4n}$ be the group of symmetries of a regular $4n$-gon. Then $\mathrm{Dih}_{4n}$ has a subgroup $H$ of index $2$ which is isomorphic to $\mathrm{Dih}_{2n}$ (these are the symmetries preserving \InnaB{a regular $2n$-gon obtained by removing half of the vertices}).
 
 Let $\rho$ be the tautological $2$-dimensional representation of $\mathrm{Dih}_{4n}$. Restricted to $H$, this gives the tautological representation of $\mathrm{Dih}_{2n}$. By the construction in Section \ref{sssec:constr_forests}, we obtain a McKay graph for the $2$-dimensional representation of $G'=\mathrm{Dih}_{4n} \ltimes \InnaB{C}_3$ coming from $\rho$ \InnaB{satisfying} $$\Gamma(G', S(\rho)) \;=\; \Gamma(\mathrm{Dih}_{4n}, \rho) \,\bigsqcup\, \Gamma(\mathrm{Dih}_{2n}, \rho\downarrow_H).$$
 
 Thus $\Gamma(G', S(\rho))$ is a \InnaB{disjoint} union of $2$ non-isomorphic trees $\tilde{D}_{\InnaB{4}n}$ and $\tilde{D}_{\InnaB{2}n}$,
where one has  $\InnaB{2n} + 3$ vertices and the other $\InnaB{n} + 3$ vertices.

A similar construction is possible for the binary dihedral groups $\mathrm{BDih}_{2n}$ and $\mathrm{BDih}_n$.
\end{example}

\InnaB{
In order to construct the examples in this section, we will need the following results on the dihedral group $\mathrm{Dih}_2 \cong C_2 \times C_2$, the tetrahedral group  $\mathbf{T}$ and the octrahedral group $\mathbf{O}$, all seen as subgroups of $SO(3, \mathbb{R})$. These statements are well-known, but we include their proofs for completeness of the presentation.

\begin{lemma}\label{lem:normal_tower}
The groups $\mathrm{Dih}_2, \mathbf{T}$ are normal subgroups of $\mathbf{O}$, with quotients $\quotient{\mathbf{O}}{\mathbf{T}} = C_2$ and $\quotient{\mathbf{O}}{\mathrm{Dih}_2} = S_3$. Furthermore,  $\mathrm{Dih}_2 \triangleleft \mathbf{T}$ with $\quotient{\mathbf{T}}{\mathrm{Dih}_2} = C_3$.
\end{lemma}
\begin{remark}
The group $\mathbf{O}$ is isomorphic to the symmetric group $S_4$. Under this isomorphism, the group $\mathbf{T}$ corresponds to $A_4$ and the group $\mathrm{Dih}_2$ corresponds to the normal subgroup $\{\id, (12)(34), (13)(24), (14)(23) \}$ of $S_4$.
\end{remark}

\begin{proof}
Consider a cube in $\mathbb{R}^3$ whose edges are parallel to the coordinate axes. We will number its vertices as follows:
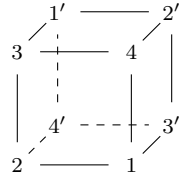
\begin{equation*}
    \begin{tikzpicture}[anchorbase,scale=1.5]
\draw[-] (0,0.2)--(0,0.8) ;
\draw[-] (0.2,0)--(0.8,0) ;
\draw[-] (1,0.2)--(1,0.8) ;
\draw[-] (0.2,1)--(0.8,1) ;

\node at (0,0) {$ \scriptstyle 2 $};
\node at (0,1) {$\scriptstyle 3 $};
\node at (1,0) {$\scriptstyle 1 $};
\node at (1,1) {$\scriptstyle 4 $};

\draw[-, dashed] (0.1,0.1)--(0.253553391,0.253553391) ;
\draw[-] (0.1,1.1)--(0.253553391,1.253553391) ;
\draw[-] (1.1,0.1)--(1.253553391,0.253553391) ;
\draw[-] (1.1,1.1)--(1.253553391,1.253553391) ;

\node at (0.353553391,0.353553391) {$\scriptstyle 4'$};
\node at (1.353553391,0.353553391) {$ \scriptstyle 3' $};
\node at (0.353553391,1.353553391) {$ \scriptstyle 1' $};
\node at (1.353553391,1.353553391) {$ \scriptstyle 2'$};

\draw[-, dashed] (0.353553391,0.553553391)--(0.353553391,1.153553391) ;
\draw[-, dashed] (0.553553391,0.353553391)--(1.153553391,0.353553391) ;
\draw[-] (1.353553391,0.553553391)--(1.353553391,1.153553391) ;
\draw[-] (0.553553391,1.353553391)--(1.153553391,1.353553391) ;
\end{tikzpicture} 
\end{equation*}
The vertices $1, 2', 3,4'$ form a regular tetrahedron ``inscribed" in the cube, and the same goes for the vertices $1',2, 3', 4 $:
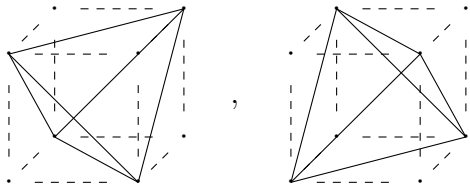
\begin{equation*}
    \begin{tikzpicture}[anchorbase,scale=1.7]
\draw[-, dashed] (0,0.2)--(0,0.8) ;
\draw[-, dashed] (0.2,0)--(0.8,0) ;
\draw[-, dashed] (1,0.2)--(1,0.8) ;
\draw[-, dashed] (0.2,1)--(0.8,1) ;

\node at (0,0) {$ \cdot $};
\node at (0,1) {$\cdot$};
\node at (1,0) {$\cdot$};
\node at (1,1) {$\cdot $};

\draw (1,0)--(0.353553391,0.353553391) ;
\draw (1,0)--(1.353553391,1.353553391) ;
\draw (1,0)--(0,1) ;
\draw (0,1)--(0.353553391,0.353553391) ;
\draw (0,1)--(1.353553391,1.353553391) ;
\draw (0.353553391,0.353553391)--(1.353553391,1.353553391) ;

\draw[-, dashed] (0.1,0.1)--(0.253553391,0.253553391) ;
\draw[-, dashed] (0.1,1.1)--(0.253553391,1.253553391) ;
\draw[-, dashed] (1.1,0.1)--(1.253553391,0.253553391) ;
\draw[-, dashed] (1.1,1.1)--(1.253553391,1.253553391) ;

\node at (0.353553391,0.353553391) {$\cdot $};
\node at (1.353553391,0.353553391) {$ \cdot $};
\node at (0.353553391,1.353553391) {$ \cdot $};
\node at (1.353553391,1.353553391) {$\cdot$};

\draw[-, dashed] (0.353553391,0.553553391)--(0.353553391,1.153553391) ;
\draw[-, dashed] (0.553553391,0.353553391)--(1.153553391,0.353553391) ;
\draw[-, dashed] (1.353553391,0.553553391)--(1.353553391,1.153553391) ;
\draw[-, dashed] (0.553553391,1.353553391)--(1.153553391,1.353553391) ;
\end{tikzpicture} 
\quad , \quad 
    \begin{tikzpicture}[anchorbase,scale=1.7]
\draw[-, dashed] (0,0.2)--(0,0.8) ;
\draw[-, dashed] (0.2,0)--(0.8,0) ;
\draw[-, dashed] (1,0.2)--(1,0.8) ;
\draw[-, dashed] (0.2,1)--(0.8,1) ;

\node at (0,0) {$ \cdot $};
\node at (0,1) {$\cdot$};
\node at (1,0) {$\cdot$};
\node at (1,1) {$\cdot $};

\draw[-, dashed] (0.1,0.1)--(0.253553391,0.253553391) ;
\draw[-, dashed] (0.1,1.1)--(0.253553391,1.253553391) ;
\draw[-, dashed] (1.1,0.1)--(1.253553391,0.253553391) ;
\draw[-, dashed] (1.1,1.1)--(1.253553391,1.253553391) ;

\draw (0,0)--(0.353553391,1.353553391) ;
\draw (0,0)--(1.353553391,0.353553391) ;
\draw (0,0)--(1,1) ;
\draw (1,1)--(0.353553391,1.353553391) ;
\draw (1,1)--(1.353553391,0.353553391) ;
\draw (0.353553391,1.353553391)--(1.353553391,0.353553391) ;

\node at (0.353553391,0.353553391) {$\cdot $};
\node at (1.353553391,0.353553391) {$ \cdot $};
\node at (0.353553391,1.353553391) {$ \cdot $};
\node at (1.353553391,1.353553391) {$\cdot$};

\draw[-, dashed] (0.353553391,0.553553391)--(0.353553391,1.153553391) ;
\draw[-, dashed] (0.553553391,0.353553391)--(1.153553391,0.353553391) ;
\draw[-, dashed] (1.353553391,0.553553391)--(1.353553391,1.153553391) ;
\draw[-, dashed] (0.553553391,1.353553391)--(1.153553391,1.353553391) ;
\end{tikzpicture} 
\end{equation*}

These are the only two regular tetrahedrons whose vertices are vertices of the cube. The group of symmetries of the cube, $\mathbf{O}$, acts on the set of two tetrahedrons; let us denote such the stabilizer subgroup of one of the tetrahedrons by $H\subset \mathbf{O}$. The action is clearly transitive, so $[ \mathbf{O}:H]=2$ and thus $H \triangleleft \mathbf{O}$, and $\mathbf{O}/H \cong C_2$. This implies that $\abs{H}=12$. Furthermore, the subgroup $H$ stabilizes both tetrahedrons, which means that $H$ acts on each tetrahedron by isometries which preserve the tetrahedron and the orientation. Thus $H \subset \mathbf{T}$, and since $\abs{H}=12$, we conclude that $H = \mathbf{T}$. 

Next, consider the centers of the faces in the cube. These $6$ points give us $3$ mutually orthogonal segments (they are depicted below); we will denote the set of these segments by $\mathcal{S}$.

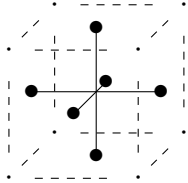
\begin{equation*}
    \begin{tikzpicture}[anchorbase,scale=1.7]
\draw[-, dashed] (0,0.2)--(0,0.8) ;
\draw[-, dashed] (0.2,0)--(0.8,0) ;
\draw[-, dashed] (1,0.2)--(1,0.8) ;
\draw[-, dashed] (0.2,1)--(0.8,1) ;

\node at (0,0) {$ \cdot $};
\node at (0,1) {$\cdot$};
\node at (1,0) {$\cdot$};
\node at (1,1) {$\cdot $};

\draw[-, dashed] (0.1,0.1)--(0.253553391,0.253553391) ;
\draw[-, dashed] (0.1,1.1)--(0.253553391,1.253553391) ;
\draw[-, dashed] (1.1,0.1)--(1.253553391,0.253553391) ;
\draw[-, dashed] (1.1,1.1)--(1.253553391,1.253553391) ;

\node at (0.5,0.5) {$ \bullet $};
\node at (0.6752,0.1752) {$\bullet$};
\node at (0.1752, 0.6752) {$\bullet$};
\node at (0.75, 0.75) {$\bullet$};
\node at (1.1752, 0.6752) {$\bullet$};
\node at (0.6752, 1.1752) {$\bullet$};

\draw (0.5,0.5)--(0.75, 0.75) ;
\draw (0.6752,0.1752)--(0.6752, 1.1752);
\draw (0.1752, 0.6752)-- (1.1752, 0.6752);

\node at (0.353553391,0.353553391) {$\cdot $};
\node at (1.353553391,0.353553391) {$ \cdot $};
\node at (0.353553391,1.353553391) {$ \cdot $};
\node at (1.353553391,1.353553391) {$\cdot$};

\draw[-, dashed] (0.353553391,0.553553391)--(0.353553391,1.153553391) ;
\draw[-, dashed] (0.553553391,0.353553391)--(1.153553391,0.353553391) ;
\draw[-, dashed] (1.353553391,0.553553391)--(1.353553391,1.153553391) ;
\draw[-, dashed] (0.553553391,1.353553391)--(1.153553391,1.353553391) ;
\end{tikzpicture} 
\end{equation*}
These segments also connect the midpoints of opposite edges in each of the \InnaC{two} tetrahedrons above.

The group $\mathbf{O}$ and its subgroup $\mathbf{T}$ act on the set $\mathcal{S}$. Let $H'\triangleleft \mathbf{O}$ be the kernel of this action, i.e. the intersection of all the stabilizers; clearly, $H'$ consists of the rotations by multiples of $\pi$ radians through any of the axes in $\mathcal{S}$. So we have: $\abs{H'}=4$, $H'\subset \mathbf{T}$ and
$[\mathbf{O}:H'] = 6$. The action of $\mathbf{O}$ on the set $\mathcal{S}$ gives a homomorphism $\mathbf{O}/H' \hookrightarrow S_3$, and since $[\mathbf{O}:H'] = 6$, we conclude that this is an isomorphism. Computing the index of $H'$ in $\mathbf{T}$ we conclude: $\mathbf{T}/H' \cong C_3$. 

It remains to check that $H'=\mathrm{Dih}_2 \cong C_2 \times C_2$, which follows from the description of $\mathrm{Dih}_2$ as a subgroup of $SO(3, \mathbb{R})$ (see Example \ref{ex:Dih_2}). The lemma is proved.
\end{proof}

Lifting the results of Lemma \ref{lem:normal_tower} via the double cover $SU(2) \to SO(\InnaC{3}, \mathbb{R})$, we immediately obtain a similar relation between the binary tetrahedral group $ \mathrm{B}\mathbf{T}$, the binary octahedral group $ \mathrm{B}\mathbf{O}$, and the binary dihedral group $\mathrm{BDih}_2 \cong Q_8$:

\begin{corollary}\label{cor:normal_tower_binary_poly}
The groups $\mathrm{BDih}_2, \mathrm{B}\mathbf{T}$ are normal subgroups of $\mathrm{B}\mathbf{O}$, with quotients $\quotient{\mathrm{B}\mathbf{O}}{\mathrm{B}\mathbf{T}} = C_2$ and $\quotient{\mathrm{B}\mathbf{O}}{\mathrm{BDih}_2} = S_3$. Furthermore, we have $\mathrm{BDih}_2 \triangleleft \mathrm{B}\mathbf{T}$ with $\quotient{\mathrm{B}\mathbf{T}}{\mathrm{BDih}_2} = C_3$.
\end{corollary}

}

\begin{example}

 Let $\mathrm{B}\mathbf{T}$ be the binary tetrahedral group, \InnaB{and consider the normal subgroup $\mathrm{BDih}_2 \triangleleft \mathrm{B}\mathbf{T}$ seen in Corollary \ref{cor:normal_tower_binary_poly}.
 
  Let $\rho$ be the tautological $2$-dimensional representation of $\mathrm{B}\mathbf{T}\subset SU(2)$. Then $\rho\downarrow_{\mathrm{BDih}_2}$ is the $2$-dimensional irreducible representation of $\mathrm{BDih}_2 \cong Q_8$.}
  
  The quotient $\quotient{\mathrm{B}\mathbf{T}}{\mathrm{BDih}_2} \cong C_3$ acts on the Klein $4$-group $\mathbb{F}_2^2$ by group automorphisms permuting transitively the order $2$ elements; the action is given by the embedding $$\quotient{\mathrm{B}\mathbf{T}}{\mathrm{BDih}_2} \cong C_3 \hookrightarrow GL_3(\mathbb{F}_2) \cong S_3.$$

 By Section \ref{sssec:constr_forests} and Remark \ref{rmk:constr_forests}, we obtain a McKay graph for the $2$-dimensional representation of $G'=\mathrm{B}\mathbf{T}\ltimes \mathbb{F}_2^2$ coming from $\rho$: $$\Gamma(G', S(\rho)) \; =\; \Gamma(\mathrm{B}\mathbf{T}, \rho) \,\bigsqcup \,\Gamma(Q_8, \rho\downarrow_{Q_8})$$
 
 Thus $\Gamma(G', S(\rho))$ is a disjoint union of $2$ non-isomorphic trees $\widetilde{E}_6$ and $\widetilde{D}_4$.
 
 \end{example}
 
\begin{example}
 Let $\mathrm{B}\mathbf{O}$ be the binary octahedral group, and consider the subgroup $ \mathrm{B}\mathbf{T} \triangleleft \mathrm{B}\mathbf{O}$ given by Corollary \ref{cor:normal_tower_binary_poly}.
 
 Let $\rho$ be the tautological $2$-dimensional representation of $\mathrm{B}\mathbf{O}\subset SU(2)$. Then $\rho\downarrow_{\mathrm{B}\mathbf{T}}$ is the tautological $2$-dimensional irreducible representation of $\mathrm{B}\mathbf{T}$.
 
By Section \ref{sssec:constr_forests}, we obtain a McKay graph for the $2$-dimensional representation of $G'=\mathrm{B}\mathbf{O} \ltimes \InnaB{C_3}$ coming from $\rho$: $$\Gamma(G', S(\rho)) \;=\; \Gamma(\mathrm{B}\mathbf{O}, \rho) \,\bigsqcup \,\Gamma(\mathrm{B}\mathbf{T}, \rho\downarrow_{\mathrm{B}\mathbf{T}})$$
 
 Thus $\Gamma(G', S(\rho))$ is a disjoint union of $2$ non-isomorphic trees $\widetilde{E}_7$ and $\widetilde{E}_6$.
\end{example}

 \begin{example}
   Let $G=\mathrm{B}\mathbf{O}$ and 
   and consider the subgroup $\mathrm{BDih}_2 \triangleleft \mathrm{B}\mathbf{O}$ given by Corollary \ref{cor:normal_tower_binary_poly}.
   
   Let $\rho$ be the \InnaB{tautological} $2$-dimensional representation of $\mathrm{B}\mathbf{O}\subset SU(2)$. Then $\rho\downarrow_{\mathrm{BDih}_2}$ is the tautological $2$-dimensional irreducible representation of $\mathrm{BDih}_2 \cong Q_8 $.

   \InnaB{Recall that by Corollary \ref{cor:normal_tower_binary_poly}, $\quotient{\mathrm{B}\mathbf{O}}{\mathrm{BDih}_2} \cong S_3$.} Then $\quotient{\mathrm{B}\mathbf{O}}{\mathrm{BDih}_2}$ acts on the Klein $4$-group $\mathbb{F}_2^2$ by group automorphisms permuting transitively the order $2$ elements; the action is given by the isomorphism $$\quotient{\mathrm{B}\mathbf{O}}{\mathrm{BDih}_2} \cong S_3 \stackrel{\sim}{\longrightarrow} GL_3(\mathbb{F}_2) .$$

 By Section \ref{sssec:constr_forests} and Remark \ref{rmk:constr_forests}, we obtain a McKay graph for the $2$-dimensional representation of $G'=\mathrm{B}\mathbf{O}\ltimes \mathbb{F}_2^2$ coming from $\rho$: $$\Gamma(G', S(\rho)) \;=\; \Gamma(\mathrm{B}\mathbf{O}, \rho) \,\bigsqcup\, \Gamma(Q_8, \rho\downarrow_{Q_8})$$
 
 Thus $\Gamma(G', S(\rho))$ is a disjoint union of $2$ non-isomorphic trees $\widetilde{E}_7$ and $\widetilde{D}_4$.
 \end{example}
}

\end{document}